\documentclass[12pt]{amsart}    
\usepackage{graphicx}
\usepackage{amsthm}
\usepackage{amsmath,amsfonts,amssymb,mathrsfs}
\usepackage{color}
\usepackage{epstopdf}
\usepackage{subfigure}
\usepackage[utf8]{inputenc}
\usepackage{hyperref}
\usepackage{cleveref}
\usepackage{cite}
\usepackage{bookmark}
\newtheorem{theorem}{Theorem}[section]
\newtheorem{lem}[theorem]{Lemma}
\newtheorem{cor}[theorem]{Corollary}
\newtheorem{prop}[theorem]{Proposition}
\newtheorem{rem}[theorem]{Remark}
\newtheorem{exa}[theorem]{Example}
\newtheorem{defin}[theorem]{Definition}

\newcommand{\GL}{{\rm GL}}
\newcommand{\SL}{{\rm SL}}
\newcommand{\SO}{{\rm SO}}

\newcommand{\du}{{\rm d}\,}
\newcommand{\ol}{\overline}
\let\ap=\alpha
\newcommand{\C}{{\mathbb C}}
\newcommand{\R}{{\mathbb R}}

\newcommand{\Z}{{\mathbb Z}}
\newcommand{\N}{{\mathbb N}}
\newcommand{\Hm}{{\mathbb H}}

\newcommand{\G}{{\mathcal G}}

\let\mi=\setminus

\let\ap=\alpha

\begin{document}

	\title{Cartan subgroups in connected locally compact groups}
	
	\author[A.\ Mandal]{Arunava Mandal}
	\address{Department of Mathematics,
		Indian Institute of Technology, Roorkee,
		Uttarakhand 247667, India}  
	\email{arunava@ma.iitr.ac.in, a.arunavamandal@gmail.com} 
	\author[R.\ Shah]{Riddhi Shah}
	\address{School of Physical Sciences, Jawaharlal Nehru University, New Delhi 110016,
		India}  
	\email{rshah@jnu.ac.in, riddhi.kausti@gmail.com} 
	
	
	
	\date{}

\maketitle

\begin{abstract} We define Cartan subgroups in connected locally compact groups, which extends the classical 
	notion of Cartan subgroups in Lie groups. We prove their existence and justify our choice of the definition 
	which differs from the one given by Chevalley on general groups. Apart from proving some properties of 
	Cartan subgroups, we show that the Cartan subgroups of the quotient groups are precisely the images of 
	Cartan subgroups of the ambient group. We establish the so-called `Levi' decomposition of Cartan 
	subgroups which extends W\"ustner's decomposition theorem and our earlier results for Lie groups. 
	We also show that the centraliser of any maximal torus of the radical is connected and its Cartan subgroups 
	are also Cartan subgroups of the ambient group; moreover, every Cartan subgroup arises this way. 
	We prove that Cartan subalgebras defined by Hofmann and Morris in pro-Lie algebras are the 
	same as those corresponding to Cartan subgroups in case of pro-Lie algebras of connected locally 
	compact groups, and that they are nilpotent. We characterise density of the image of a power map 
	in a connected locally compact group in terms of its surjectivity on all Cartan subgroups, and show that 
	weak exponentiality of the group is equivalent to the condition that all its Cartan subgroups are connected.
\end{abstract}

	\section{Introduction}
A Cartan subgroup is a classical object in Lie theory defined by C.\ Chevalley \cite{Ch} in 1955. Chevalley defined a purely 
intrinsic notion of Cartan subgroups $C$ of an abstract group $G$; it is a maximal nilpotent subgroup of $G$ with the additional 
property that every subgroup which is normal and of finite index in $C$ is also of finite index in its own normaliser in $G$. It is shown 
in \cite{Ch} that Cartan subgroups of a Zariski connected algebraic group $G$ are exactly those irreducible algebraic subgroups of $G$ 
whose Lie algebras are the Cartan subalgebras of the Lie algebra of $G$. Note that Cartan subgroups of a connected Lie group $G$ 
are closed in $G$, but need not be connected, and their Lie algebras are the Cartan subalgebras of the Lie algebra of $G$; more precisely, 
each Cartan subalgebra is the Lie algebra of exactly one Cartan subgroup. The structural aspect of Cartan subgroups have been studied 
extensively in algebraic groups or Lie groups by Borel \cite{Bo}, Goto \cite{Go}, Togo \cite{To}, Prasad and Raghunathan \cite{PRagh}, 
and many others. Cartan subgroups also play a major role in the study of weak exponentiality and power maps 
(see e.g. \cite{HoMu, N, Wu1, DM, BM, M1, M2, MS}).

In this article, we explore the appropriate notion of Cartan subgroups in a connected locally 
compact group which is consistent with the well-established definition of Cartan subgroups in a 
connected Lie group. Our definition of Cartan subgroups slightly differs from Chevalley's definition 
mentioned in the beginning. To put it in a suitable framework, let us recall the other well-known 
definition of Cartan subgroups in a connected Lie group given in \cite{N}.

Let $G$ be a connected Lie group with the Lie algebra $\G$. Let $\mathcal C$ be a Cartan subalgebra of $\G$, i.e.\ it is a (maximal) 
nilpotent Lie subalgebra which coincides with its own normaliser. Let  $\mathcal C_\C$ (resp.\ $\mathcal G_\C$) be the complexification of $\mathcal C$ 
(resp.\  $\G$) and let $\Delta$ be the set of roots of $\G_\C$ belonging to $\mathcal C_\C$. Then 
$\G_\C= \mathcal C_\C+\Sigma_{\alpha\in\Delta}\G_ \C^{\alpha}$. Recall that $N_G(\mathcal C):=\{g\in G\mid{\rm Ad}(g)(\mathcal C)=  \mathcal C\}.$  

\begin{defin} [\cite{N}]\label{Neeb}
	A closed subgroup $C$ of a connected Lie group $G$ is called a Cartan subgroup if the following hold:
	\begin{enumerate}
		\item[{${\rm (I)}$}] The Lie algebra of $C$, denoted by $\mathcal C$, is a Cartan subalgebra of $\G$.
		\item[{${\rm (II)}$}] $C(\mathcal C):=
		\{g\in N_G(\mathcal C)\mid\alpha\circ{\rm Ad}(g)|_{\mathcal C_\C}=\alpha \mbox{ for all }  \alpha\in\Delta\}=C$. 	
	\end{enumerate}
\end{defin}

Definition \ref{Neeb} is equivalent to the definition given by Chevalley on any connected Lie group \cite{N}. The absence of root space 
decomposition in a connected locally compact group is a primary obstruction to extending Definition \ref{Neeb}, whereas, {\it a priori}, 
there is no problem to consider Chevalley's definition of Cartan subgroups as it is in the context of connected locally compact groups, 
since it is purely group-theoretic. However, Chevalley's definition does not work even for compact connected (non-Lie) groups as illustrated 
by Example \ref{cpt-u} below. This suggests that Chevalley's definition of Cartan subgroups needs to be modified 
in case of general connected locally compact groups. In this framework, we define Cartan subgroup as
a maximal nilpotent subgroup, with the modified condition that for any closed normal subgroup $L$ of $C$, if $C/L$ is
compact and totally disconnected, then $N_G(L)/L$ is compact and totally disconnected (see Definition \ref{Cartan in l.c}).

First, we show the existence of Cartan subgroups in any connected locally compact group $G$ (see Theorem \ref{existence-cartan})
and give examples of Cartan subgroups in $G$ which are not Lie groups (see Examples \ref{cpt-u}, \ref{non-Lie} and \ref{non-ab2}).  
We prove that every Cartan subgroup $C$ is a projective limit of Cartan subgroups in Lie groups (see Proposition \ref{mod-cpt}). There 
is a notion of Cartan subalgebras of a pro-Lie algebra introduced by Hofmann and Morris \cite{HoMo2} which are pro-nilpotent. Here, 
we prove that the Lie algebra of a Cartan subgroup is a Cartan subalgebra of the pro-Lie algebra of the connected locally compact 
group. In fact, we show that there is a one-to-one correspondence between Cartan subalgebras and Cartan subgroups (see Theorem 
\ref{cartan-subalgebra}). This implies in particular that the Cartan subalgebras in this case are nilpotent.

A connected locally compact group $G$ is a projective limit of connected Lie groups; this fact is part of the solution of 
Hilbert's fifth problem \cite{Y}. Therefore, 
one expects to extend various structural results about Cartan subgroups from Lie groups to all connected locally compact groups. 
Here, we obtain various characterisations of Cartan subgroups, extending our recent results obtained in \cite{MS}. One would 
expect to prove results on such groups $G$ by a technique of  passing through the limit and using the results on Lie groups at hand. 
Although this technique works in some cases, however, to a large extent, it does not help to prove many of our results. Instead, 
we need to use various structural results on connected locally compact groups, Levi decompositions, and reduction processes.

As in the case of connected Lie groups, Cartan subgroups of a connected locally compact group which is a compact extension 
of a closed connected solvable normal subgroup turn out to be connected, and conjugate to each other. For connected compact groups, 
they are precisely the maximal connected abelian subgroups (see Proposition \ref{abelian-conj}). We also give examples of connected 
locally compact groups with non-abelian Cartan subgroups (see Examples \ref{non-ab1} and \ref{non-ab2}). 
Unlike in the case of Lie groups, a connected locally compact group $G$ can be an arbitrary product of nontrivial connected locally 
compact groups $G_\alpha$, in which all but finitely many are compact; then the arbitrary (infinite) product of Cartan subgroups 
$C_\alpha$ of  $G_\alpha$ is again a Cartan subgroup of $G$. Conversely, each Cartan subgroup of $G$ is of this form 
(see Proposition \ref{product}). 

Now we list some important generalisations from connected Lie groups to locally compact groups. In \S\,5, we characterise 
Cartan subgroups in quotient groups. More precisely, if $H$ is any closed normal subgroup of a connected 
locally compact group $G$, then for a Cartan subgroup  $C$ of $G$, the image $CH/H$ is a Cartan subgroup of $G/H$. 
Conversely, any Cartan subgroup in the quotient group $G/H$ is the image of some Cartan subgroup in $G$ 
(see Theorem \ref{quotient}). This is instrumental in deducing one of our main important results on a `Levi' decomposition 
of Cartan subgroups (see \S\,6). To put it in perspective, let us recall that there is a Levi decomposition for 
a connected locally compact group $G$,  namely, $G = SR$, where $R$ is the radical of $G$, and $S$ is either trivial or $S$ is 
a Levi subgroup, i.e.\ a maximal connected semisimple subgroup of $G$ (see \cite{Mat}, and \S\,6 for details).  Generalising 
W\"ustner's decomposition theorem \cite{Wu1} on Lie groups and Theorem 1.1 of \cite{MS}, we prove that given a 
Cartan subgroup $C$ of $G$, there exists a Levi decomposition $G=SR$ such that $C=(C\cap S)(C\cap R)$ with additional 
properties that  $(i)$ $C\cap S$ is a Cartan subgroup of $S$, $(ii)$ elements of $C\cap S$ centralise $C\cap R$, and $(iii)$ 
$Z_R(C\cap S)$, the centraliser of $C\cap S$ in $R$, is a closed connected solvable group and $C\cap R$ is a Cartan subgroup 
of $Z_R(C\cap S)$, which in particular implies that $C\cap R$ is connected. Conversely, given any Levi subgroup $S$ of $G$ and 
a Cartan subgroup $C_S$ of $S$, $Z_R(C_S)$ is connected and for any Cartan subgroup $C_{Z_R(C_S)}$ of $Z_R(C_S)$, if we 
set $C:=C_SC_{Z_R(C_S)}$, then $C$ is a Cartan subgroup of $G$ (see Theorem \ref{decomp}). Moreover, we prove that 
$C=Z_G(C^0)C^0$ in Corollary \ref{centraliser-conn}, which reiterates the one-to-one correspondence between Cartan subgroups 
and Cartan subalgebras in connected locally compact groups, as is the case in Lie groups. We also characterise Cartan subgroups 
of $G$ in terms of Cartan subgroups of $Z_G(T_R)$, where $T_R$ is any maximal compact subgroup of 
the radical $R$ in $G$, and $Z_G(T_R)$ is the centraliser of $T_R$ in $G$, which is shown to be connected.  
We show that each Cartan subgroup of $Z_G(T_R)$ is a Cartan subgroup of $G$, and conversely, any Cartan 
subgroup $C$ of $G$ is a Cartan subgroup of $Z_G(T_R)$ for some $T_R\subset C$ (see Theorem \ref{cartan-torus}). 

As an application, we give a characterisation of the density of the image of a power map on a locally compact group 
through Cartan subgroups in Theorem \ref{power-dense} for which we use Proposition \ref{pk-quo}, both of these generalise earlier 
results on Lie groups in this context (see \cite{BM} and \cite{MS}). We also correlate the density of images of power maps with weak
exponentiality for connected locally compact groups (see Corollary \ref{C-power-cartan}); the latter property has 
attracted much attention in the literature and is studied by many mathematicians, see \cite{HoMu, N, BM, M1, M2, MS, MR}. 

\section{Prerequisites and notations} 

Throughout the paper, we use the following notations. Let $G$ be a locally compact (Hausdorff) group and let $H$ be a subgroup of $G$. 
Let $H^0$ denote the connected component of the identity in $H$, $D(H)=\ol{[H,H]}$, the closure of the commutator subgroup 
$[H,H]$ in $H$ and let $Z(H)$ denote the center of $H$, $Z^0(H)=(Z(H))^0$, let $N_G(H)$ denote 
the normaliser of $H$ in $G$ and let $N_G^0(H)=(N_G(H))^0$. All these subgroups are closed (resp.\ characteristic) in $G$ 
if $H$ is closed (resp.\ characteristic) in $G$. For a subset $B$ of $G$, let $\ol{B}$ denote the closure of $B$ in $G$. 
Then $\ol{B}$ is a subgroup if $B$ is so, and it is connected (resp.\ normal in $G$) if $B$ is connected (resp.\ normal in $G$). Let $Z_H(B)$ 
denote the centraliser of the set $B$ in the subgroup $H$ of $G$. It is a closed subgroup of $H$ and let $Z^0_H(B)=(Z_H(B))^0$. If $B$ is a 
subgroup of $G$, then let $N_H(B)=N_G(B)\cap H$ and $N^0_H(B)=(N_H(B))^0$. 
For a nilpotent group $G$, let $l(G)$ denote the length of the central series of $G$.

Any connected locally compact group $G$ has a largest connected solvable normal subgroup which is known as its radical
(see Theorem 15 in \cite{I}, see also Theorem 6.11 in \cite{Gl}). Note that the closure of a solvable group is solvable, and hence the 
radical is closed in $G$. If the radical of $G$ is trivial and $G$ is nontrivial, then $G$ is said to be semisimple.

It follows from a classical theorem of Yamabe \cite{Y} that every connected locally compact group is a projective limit of Lie groups. 
The case of solvable groups was proven earlier by Iwasawa \cite{I} as well as Gleason \cite{Gl}.

\begin{theorem} {\rm \cite[Theorem 5$'$]{Y}} \label{Yam-Glea}
	Let $G$ be a connected locally compact group. Then there exists $\{K_\alpha\}_{\alpha\in\Lambda}$, 
	a collection of compact normal subgroups in $G$, such that $G/K_\alpha$ is a Lie group for each 
	$\alpha$ and $\cap_{\alpha\in\Lambda} K_\alpha=\{e\}$,  i.e.\ $G=\varprojlim_\ap G/K_\alpha$.
\end{theorem}

	\begin{rem} \label{Kap} For a connected locally compact group $G$, let $\{K_\alpha\}_{\alpha\in\Lambda}$
	be a collection of compact normal subgroups in $G$ as in the above theorem, for $\alpha\in\Lambda$ denote by 
	$G_\alpha:=G/K_\alpha$, the quotient group, which is a Lie group, and by $\pi_\ap:G\to G_\alpha$ the natural projection.  
	So $G=\varprojlim_{\ap\in\Lambda} G_\alpha$, the projective limit of Lie groups $G_\alpha$. We may  often drop the 
	indexing set $\Lambda$.
\end{rem}

Throughout, we will use the notations as in Remark \ref{Kap} when we mention $K_\ap$, $G_\ap$ and $\pi_\ap$ for a connected locally 
compact group $G$.

\section{Cartan subgroups in a connected locally compact group}

In this section, we define the notion of Cartan subgroups in the context of connected locally compact groups. 
We also prove the existence of Cartan subgroups in a connected locally compact group.

\begin{defin} \label{Cartan in l.c}
	A subgroup $C$ of a connected locally compact group $G$ is said to be a {\it Cartan} subgroup of $G$ if the following conditions
	are satisfied:
	\begin{enumerate}
		\item[{${\rm (I)}$}] $C$ is a maximal nilpotent subgroup of $G$.
		\item[{${\rm (II)}$}] If $L$ is any closed normal subgroup of $C$ such that 
		$C/L$ is compact and totally disconnected, then 
		$N_G(L)/L$ is compact and totally disconnected, where $N_G(L)$ is the normaliser of $L$ in $G$.
	\end{enumerate}	
\end{defin}

\begin{rem}
	If we take $L=C$ in Definition \ref{Cartan in l.c}, we get that $N_G(C)/C$ is compact and totally disconnected. In particular, 
	if $C$ is normal in $G$, then $G/C$ is both connected and totally disconnected, i.e.\ $G=C$. Therefore, we have that $C$ 
	is a normal Cartan subgroup of $G$ if and only $G=C$, if and only if $G$ is nilpotent.
\end{rem}

\begin{rem}\label{Remark on Cartan}
	Note that if $G$ is a connected Lie group, then the quotients $C/L$ and $N_G(L)/L$ are Lie groups and any compact 
	totally disconnected Lie group is discrete and finite. Therefore, Definition \ref{Cartan in l.c} for Cartan subgroups extends 
	the one given by Chevalley in the case of Lie groups. We will also show in Proposition \ref{mod-cpt}, that in a connected 
	locally compact group,  every Cartan subgroup arises as a projective limit of Cartan subgroups of Lie groups. 
	Example \ref{cpt-u} will justify our choice of condition {\rm (II)} for Cartan subgroups in connected locally compact 
	{\rm (}non-Lie{\rm )} groups, as in a compact connected infinite dimensional group $G$, which is a product of 
	certain semisimple Lie groups, the maximal nilpotent group is a  compact connected infinite dimensional 
	abelian group which has infinite index in its normaliser and the quotient is totally disconnected. 
\end{rem}

As the closure of a nilpotent group is a nilpotent group, every Cartan subgroup is closed. Since it is a maximal nilpotent group in $G$, 
it contains the center $Z(G)$ of $G$. 

We will prove the existence of a Cartan subgroup in a connected locally compact group. Before that we prove some preliminary 
results which are needed for this purpose. 
The following lemma is known in the case of connected Lie groups. 

\begin{lem}\label{central-cartan}
	Let $G$ be a connected locally compact group. Let $Z$ be a closed central subgroup of $G$ and let $C$ be a Cartan 
	subgroup of $G$. Then the following hold:
	
	\begin{enumerate} 
		\item[{$(1)$}] $Z\subset C$. 
		\item[{$(2)$}] $C/Z$ is a Cartan subgroup of $G/Z$. Conversely, for any subgroup $C'$ which contains $Z$, if $C'/Z$ is a 
		Cartan subgroup of $G/Z$, then $C'$ is a Cartan subgroup of $G$.
		\item[{$(3)$}] $Z_n\subset C$ and $C/Z_n$ is a Cartan subgroup of $G/Z_n$ for closed normal subgroups $Z_n$ 
		of $G$, where $Z_0=Z(G)$, $Z_n\subset Z_{n+1}$ and  $Z_{n+1}/Z_n$ is the center of $G/Z_n$, $n\in\N\cup\{0\}$. 
	\end{enumerate}
\end{lem}

\begin{proof}
	$(1):$ Since $C$ is a maximal nilpotent group, $Z\subset Z(G)\subset C$. 
	
	\smallskip
	\noindent $(2):$ Let $C$ be a Cartan subgroup of $G$. Let $\pi:G\to G/Z$ be the natural projection, 
	$C_1:=C/Z$ and let $G_1:=G/Z$. It is easy to see that $C_1$ is also a maximal nilpotent subgroup of $G_1$. 
	Let $L_1$ be a closed normal subgroup of $C_1$. Then $L_1$ is of the form $L_1=L/Z$, where $L=\pi^{-1}(L_1)$, 
	which is a closed  normal subgroup of $G$ and $Z\subset L\subset C$.  Suppose that $C_1/L_1$ is compact and totally   
	disconnected. Then $C/L$, being isomorphic to $C_1/L_1$, is also compact and totally disconnected, and hence 
	so is $N_G(L)/L$. Therefore, $\pi(N_G(L))/\pi(L)$ is compact and totally disconnected. Since $Z\subset L$, 
	we have $\pi(N_G(L))=N_{G_1}(L_1)$. Thus $N_{G/Z}(L_1)/L_1$ is compact and totally disconnected. 
	This proves that $C/Z$ is a Cartan subgroup in $G/Z$.
	
	Conversely, suppose $C'$ contains $Z$ and $C'/Z$ is a Cartan subgroup of $G/Z$. Then it is easy to see that $C'$ 
	is a maximal nilpotent subgroup of $G$. Let $L$ be a closed normal subgroup of $C'$ such that $C'/L$ is compact 
	and totally disconnected. Then $\ol{LZ}$ is a closed normal subgroup of $C'$ and $C'/\ol{LZ}$, being a quotient 
	group of $C'/L$, is compact totally disconnected. Moreover, $C'/\ol{LZ}$ is isomorphic to $(C'/Z)/(\ol{LZ}/Z)$. Since 
	$C'/Z$ is a Cartan subgroup of $G/Z$, we get that $N_{G/Z}(\ol{LZ}/Z)/\ol{(LZ}/Z)$ is compact and totally disconnected. 
	Since $N_G(\ol{LZ})/Z\subset N_{G/Z}(\ol{LZ}/Z)$ and $\ol{LZ}\subset N_G(L)\subset N_G(\ol{LZ})$, 
	we get that $N_{G}(\ol{LZ})/\ol{LZ}$, and hence $N_G(L)/\ol{LZ}$ is compact and totally disconnected. 
	As $\ol{LZ}/L$, being a subset of $C/L$, is compact and totally disconnected, so is 
	$N_G(L)/L$. This proves that $C'$ is a Cartan subgroup of $G$. 
	
	\smallskip
	\noindent $(3):$ It follows easily from (1) and (2). 
\end{proof}

We need the following crucial lemma which is deduced from some results in \cite{I} and some simple observations. 
For any subgroup $H$ of $G$, recall the notations $Z_G(H)$, $Z_G^0(H)$ and $D(H)$ defined before. These are all normal 
subgroups in $G$ if $H$ is so. If $H$ is compact and connected, then every product of commutators is a commutator, i.e.\ 
$[H,H]=\{hkh^{-1}k^{-1}\mid h,k\in H\}$, and hence 
$[H,H]$ is closed in $H$ and $D(H)=[H,H]$ (cf.\ \cite{HoMo1}, Theorem 9.2).  Recall that a nontrivial connected compact group $H$ is 
said to be semisimple if $D(H)=H$. Note that this definition is equivalent to the one for a nontrivial connected locally compact group which is 
semisimple if its radical is trivial. For any connected compact group $H$, if $D(H)$ is nontrivial then it is semisimple since $D(D(H))=D(H)$ 
(cf.\ \cite{I}, Lemma 2.2).

\begin{lem}\label{L1-cpt}
	Let $G$ be a connected locally compact group and let $K$ be a compact normal subgroup of $G$. 
	Then the following hold:
	\begin{enumerate}
		\item[{$(i)$}] $Z(K)\subset Z(G)$, $Z(K^0)\subset Z(G)$ and $Z(D(K^0))\subset Z(G)$.
		\item[{$(ii)$}]  $G=Z_G(K)K= Z_G(K)K^0=Z_G^0(K)K^0=Z^0_G(K)D(K^0)$.
		\item[{$(iii)$}] $G=Z_G(K^0)K^0=Z_G^0(K^0)K^0=Z_G(D(K^0))D(K^0)$ \\
		\hphantom{$G=Z_G(K^0)K^0=Z_G^0(K^0)K^0$} $=Z^0_G(D(K^0))D(K^0)$.
		\item[{$(iv)$}]  $K=Z(K)D(K^0)$.
		\item[{$(v)$}]  $Z_G(K)=Z_G^0(K)Z(K)=Z_G^0(K^0)Z(K^0)=Z_G^0(K)Z(D(K^0))$ \\
		\hphantom{$Z_G(K)=Z_G^0(K)Z(K)=Z_G^0(K^0)Z(K^0)$} $=Z_G^0(D(K^0))Z(D(K^0))$.
		\item[{$(vi)$}]  $K/Z(K)$ is a (compact) connected semisimple group with trivial center.
		\item[{$(vii)$}] $R\subset Z_G^0(K)$, where $R$ is the radical of $G$. 
	\end{enumerate}	
\end{lem}

\begin{proof}
	$(i):$  This follows from Theorem 4 of \cite{I} since $G$ is connected and $Z(K)$, $Z(K^0)$, $Z(D(K^0))$ are compact 
	normal abelian subgroups of $G$.  
	
	\smallskip
	\noindent $(ii):$ Theorem 2 of \cite{I} shows that $G=Z_G(K)K$. Note that as $K$ is normal in $G$, so are $Z_G(K)$ and 
	$Z_G^0(K)$. As $G$ is connected, so is $G/Z_G(K)$, and hence 
	$G=Z_G(K)K^0$. Moreover, $G/K^0$ is also connected, and we get that $G=Z_G^0(K)K^0$.  
	We know from Lemma 2.4 of \cite{I} that $K^0=Z(K^0)D(K^0)$. Therefore, $G=Z^0_G(K)Z(K^0)D(K^0)=
	Z_G(K)D(K^0)=Z_G^0(K)D(K^0)$ (using (i) and the connectedness of $G$). 
	
	\smallskip
	\noindent $(iii):$ The proof follows in the same way as in (ii), using Theorem 2 of \cite{I} and the facts that $G$ is 
	connected, and $K^0$ and $D(K^0)$ are compact normal subgroups of $G$. 
	
	\smallskip
	\noindent $(iv):$ From $(ii)$, we get that $K=(K\cap Z_G(K))D(K^0)$. Since $K\cap Z_G(K)\subset Z(K)$, it follows that  
	$K=Z(K)D(K^0)$. 
	
	\smallskip
	\noindent $(v):$  This follows from (i--ii) and the following observation: 
	$$Z_G(K)\cap D(K^0)\subset Z(K)\cap D(K^0)= Z(D(K^0)).$$ 
	
	\smallskip
	\noindent $(vi):$ As $D(K^0)$ is connected and semisimple, it follows from (iv) that $K/Z(K)$ is connected and semisimple. 
	By (i), $Z(D(K^0))\subset Z(G)$ and hence we have that $Z(K)\cap D(K^0)= Z(D(K^0))$. 
	Now by (iv), $K/Z(K)$ is isomorphic to $D(K^0)/Z(D(K^0))$. Since $D(K^0)$ is semisimple, by Theorem 9.19\,(iv) of 
	\cite{HoMo1}, we get that $D(K^0)/Z(D(K^0))$ has trivial center. Therefore, the center of $K/Z(K)$ is trivial. 
	
	\smallskip
	\noindent $(vii):$ From (ii), we have that $G=Z_G^0(K)D(K^0)$, and both $Z_G^0(K)$ and $D(K^0)$ are normal in $G$ 
	and they centralise each other. As $D(K^0)$ is semisimple, its radical is trivial and it follows that 
	$R\subset Z_G^0(K)$. 
\end{proof}

\smallskip
\noindent {\bf Existence of Cartan subgroups:} 
It is well-known that Cartan subgroups exist in connected Lie groups. We now prove the existence of Cartan subgroups in 
connected locally compact groups as in Definition \ref{Cartan in l.c}. Note again that this definition agrees with
the original definition of Cartan subgroups in case of connected Lie groups. We also justify the choice of conditions in 
the definition by an example (see Example \ref{cpt-u}).

\begin{theorem}\label{existence-cartan}
	Every connected locally compact group admits a Cartan subgroup. 
\end{theorem}

\begin{proof}
	Let $G$ be a connected locally compact group. By Lemma 4.2 of \cite{I}, $G$ admits a unique maximal compact 
	normal subgroup $K$ such that $G/K$ is a Lie group (cf.\ \cite{Y}). By Lemma \ref{L1-cpt}\,(vi), 
	$K/Z(K)$ is a connected semisimple group with trivial center. By 
	Lemma \ref{L1-cpt}\,(i), $Z(K)$ is central in $G$, and hence by Lemma \ref{central-cartan}, we may replace $G$ by 
	$G/Z(K)$ and $K$ by $K/Z(K)$, and assume that $Z(K)$ is trivial. Now $K=K^0=D(K^0)$ is connected and has 
	trivial center, and $G=Z^0_G(K)\times K$ (cf.\ Lemma \ref{L1-cpt}\,(ii, vi)). Therefore, $Z^0_G(K)$ has no nontrivial 
	compact normal subgroups, and hence it is a connected Lie group. Let $C_0$ be a Cartan subgroup in $Z^0_G(K)$. 
	Since $K$ is a connected semisimple compact group with trivial center, by Theorem 9.19\,(iv) of \cite{HoMo1}, 
	$K=\Pi_d K_d$, where each $K_d$ is a compact connected simple Lie group with trivial center. Now 
	$G=Z^0_G(K)\times \Pi_d K_d$. Let $C_d$ be a Cartan subgroup of $K_d$ for each $d$. By Proposition 6 of \cite{Wi}, 
	each $C_d$ is compact, connected and abelian, and hence so is
	$C':=\Pi_d C_d$. Let $C=C_0\times C'=C_0\times\Pi_d C_d$, where 
	$C_0$ as above is a Cartan subgroup of the Lie group $Z_G^0(K)$. We show that $C$ satisfies conditions (I) and (II) 
	and hence it is a Cartan subgroup of $G$. As $C_0$ is nilpotent and $C'$ is abelian, it follows that $C$ is nilpotent. 
	Let $\pi_0:G\to Z^0_G(K)$ and $\pi_d:G\to K_d$ be the natural projections. Then $\pi_0$ and each $\pi_d$ are open 
	continuous group homomorphisms and $G=\pi_0(G)\times \Pi_d\pi_d(G)$. Now $C$ is a 
	maximal nilpotent group as $\pi_0(C)$ and each $\pi_d(C)$ are so in $\pi_0(G)$ and $K_d$ respectively. That is, 
	$C$ satisfies condition (I).
	
	Let $L$ be a closed normal subgroup in $C$ such that $C/L$ is compact and totally disconnected. Then 
	$C^0=C_0^0\times C'=L^0$, and it follows that $L=\pi_0(L)\times C'$. Therefore, $C/L$ is isomorphic to 
	a subgroup of $C_0/C_0^0$ which is discrete, and hence $C/L$ is finite. 
	Now we get that $C_0/\pi_0(L)$ is finite. Hence $N_{\pi_0(G)}(\pi_0(L))/\pi_0(L)$ is finite. Also, 
	$\pi_d(L)=\pi_d(C^0)=C_d$ for each $d$. Therefore, $N_{K_d}(\pi_d(L))/\pi_d(L)=N_{K_d}(C_d)/C_d$ is finite, and
	hence $N^0_K(C')=C'$ and $N^0_G(L)\cap K=C'$. Now $N_G(L)=N_{\pi_0(G)}(\pi_0(L))\times\Pi_d N_{K_d}(C_d)$, 
	and it follows that $N_G(L)K/LK$ is isomorphic to $N_{\pi_0(G)}(\pi_0(L))/\pi_0(L)$ which is finite. As $K$ is compact, 
	the preceding assertion implies that $N_G(L)/L$ is compact and also that 
	$N^0_G(L)\subset L^0(N^0_G(L)\cap K)=L^0 C'=L^0$.  Therefore, $N_G(L)/L$ is totally disconnected, 
	i.e.\ $C$ satisfies condition (II) and $C$ is a Cartan subgroup of $G$.
\end{proof}

We now give an example of a connected compact group $G$, whose Cartan subgroups $C$ are connected, but $N_G(C)/C$ is 
(compact) totally disconnected and infinite. The example justifies the choice of conditions in the definition of Cartan subgroups 
in a general connected locally compact group. Note that a maximal torus in a compact connected group is a maximal connected 
abelian subgroup.

\begin{exa} \label{cpt-u}
	Take a compact connected Lie group $K$ and a maximal torus $T$ such that $N_K(T)/T$ is nontrivial and finite. 
	Many such groups exist, for example one can take $K=U(n)$, a unitary group in $\GL(n,\C)$ and   
	$T=U(1)\times\cdots\times U(1)$ $(n$-copies$)$ for any $n\geq 2$. Then $N_K(T)/T$ is isomorphic to $S_n$\ 
	$($the permutation group of $n$ elements$)$. Recall that $T$ is a Cartan subgroup of $K$. Let $G$ be an infinite 
	countable product of $K$. It is easy to see that if $C\subset G$ is an infinite countable product of $T$, then $C$ is a 
	Cartan subgroup of $G$ 
	$($see also Proposition \ref{product}$)$. Here, $N_G(C)/C$ is an infinite compact totally disconnected group.
\end{exa}

\section{Some properties of Cartan subgroups}

In this section, we establish a relation between Cartan subgroups in a connected locally compact group and those of its 
quotient groups modulo compact normal subgroups in Proposition \ref{mod-cpt}\,(1) below; we will later generalise it for
quotient groups modulo any closed normal subgroups; see Theorem \ref{quotient}. We also show in Proposition \ref{mod-cpt}\,(2a) 
below that in a connected locally compact group, every Cartan subgroup arises as a projective limit of Cartan subgroups 
of Lie groups. Proposition \ref{mod-cpt}\,(2b) will be useful in the proof of Theorem \ref{cartan-subalgebra}. 
We also list and prove some properties related to the structure of Cartan subgroups which will be useful.

	\begin{prop}\label{mod-cpt}
	Let $G$ be a connected locally compact group. Then the following hold:
	\begin{enumerate}
		\item[{$(1)$}] Let $K$ be any compact normal subgroup of $G$. If $C$ is a Cartan subgroup of $G$, then $CK/K$ 
		is a Cartan subgroup of $G/K$. Conversely, given a Cartan subgroup $C'$ in $G/K$, there exists a Cartan 
		subgroup $C$ in $G$ such that $CK/K=C'$. 
		\item[{$(2)$}] Let $G=\varprojlim_\ap G_\ap$, a projective limit of Lie groups, where  
		$G_\alpha:=G/K_\alpha$, $\{K_\alpha\}_{\ap\in\Lambda}$ and $\pi_\alpha:G\to G_\alpha$ be  
		as in Remark \ref{Kap}. Then the following hold:
		\begin{enumerate}				
			\item[{$(a)$}] $C$ is a Cartan subgroup of $G$ if and only if $\pi_\alpha(C)$ is a Cartan subgroup of $G_\alpha$ for 
			each $\alpha\in\Lambda$. 
			\item[{$(b)$}] For every pair $\ap,\beta$ in $\Lambda$ such that $K_\ap\subset K_\beta$, 
			let $\pi_{\ap\beta}: G_\ap\to G_\beta$ be the natural quotient map. Suppose $C_\ap$ is a Cartan subgroup of 
			$G_\ap$ for each $\ap\in\Lambda$ such that $\pi_{\ap\beta}(C_\ap)=C_\beta$ for all pairs $\ap,\beta$ as above. 
			Then $C:=\cap_{\ap\in\Lambda}\,\pi_\ap^{-1}(C_\ap)$ is a Cartan subgroup of $G$ and $\pi_\ap(C)=C_\ap$ 
			for every $\ap\in\Lambda$. 
		\end{enumerate}
	\end{enumerate}
\end{prop}

\begin{proof}
	$(1):$ Let $C$ be a Cartan subgroup of $G$. By Lemma \ref{L1-cpt}\,(i, vi), $Z(K)$ is central in $G$ and the center of 
	$K/Z(K)$ is trivial. As $G/K$ is isomorphic to $(G/Z(K))/(K/Z(K))$, by Lemma \ref{central-cartan}, we may replace 
	$G$ by $G/Z(K)$ and $K$ by $K/Z(K)$, and assume that $Z(K)$ is trivial. By Lemma \ref{L1-cpt}\,(vi), we have that
	$K=D(K^0)$  is connected (with trivial center), and $G=Z^0_G(K)\times K$. It is easy to see that any Cartan 
	subgroup $C$ of $G$ is of the form $C=C_1\times C_2$, where $C_1$ is a Cartan subgroup of $Z_G^0(K)$ and 
	$C_2$ is a Cartan subgroup of $K$. Hence $CK/K$, being isomorphic to $C_1$, is a Cartan subgroup of $G/K$.
	
	Conversely, let $C'$ be a Cartan subgroup of $G/K$. By Lemma \ref{L1-cpt}, $Z(K)$ is central in $G$. By 
	Lemma \ref{central-cartan}, it is enough to prove the statement for $G/Z(K)$ and $K/Z(K)$ (instead of $G$ and $K$). 
	Hence we may replace $G$ by $G/Z(K)$ and $K$ by $K/Z(K)$, and assume 
	that $K=D(K^0)$ is connected and has trivial center. By Lemma \ref{L1-cpt}, 
	$G=Z_G^0(K)\times K$. Let $C_1:=\pi^{-1}(C')\cap Z_G^0(K)$, where $\pi:G\to G/K$ is the natural projection. As the 
	restriction of $\pi$ to $Z_G^0(K)$ is an isomorphism (homomorphism and homeomorphism) onto $G/K$ and 
	$\pi(C_1)=C'$, we get that $C_1$ is a Cartan subgroup of $Z_G^0(K)$. Let $C:=C_1\times C_2$, where 
	$C_2$ is any Cartan subgroup of $K$. It is easy to see that $C$ is a Cartan subgroup of $G$ 
	and $CK/K=\pi(C)=\pi(C_1)=C'$.
	
	\smallskip
	\noindent $(2\mbox{a}):$ Let $C$ be a Cartan subgroup of $G$. Then by $(1)$, we have $\pi_\alpha(C)$ is a Cartan subgroup of 
	$\pi_\alpha(G)$ for each $\ap\in\Lambda$.
	
	Conversely, let $C$ be a subgroup of $G$ such that $\pi_\alpha(C)$ is a Cartan subgroup of $\pi_\alpha(G)=G_\alpha$ 
	for each $\alpha\in\Lambda$. As $\pi_\alpha(C)$ is a maximal nilpotent group in $G_\alpha$, it contains $Z(G_\ap)$, the 
	center of $G_\alpha$. As $\pi_\ap(Z(G))\subset Z(G_\ap)$ for each $\alpha$, we get that $Z(G)\subset C$. 
	
	Let $K$ be the maximal compact normal subgroup of $G$. By Lemma \ref{L1-cpt}\,(i), we have that $Z(K)\subset Z(G)$ and by 
	Lemma \ref{central-cartan}, it is enough to prove that $C/Z(K)$ is a Cartan subgroup of $G/Z(K)$. By Lemma \ref{L1-cpt}\,(vi),  
	$K/Z(K)$ has trivial center, and hence $G/Z(K)$ has no nontrivial compact central subgroup. We also get from (1) that for 
	each $\ap$, $(CK_\ap Z(K))/(K_\ap Z(K)$) is a 
	Cartan subgroup of $G/(K_\ap Z(K))$. Now we may replace $G$ by $G/Z(K)$ and $K_\ap$ by $(K_\ap Z(K))/Z(K)$ 
	and assume that the maximal compact normal subgroup of $G$ (which we denote by $K$ again) has trivial center. 
	By Lemma \ref{L1-cpt}\,(ii, iv), $K=D(K^0)$ is connected and has trivial center, and $G=Z_G^0(K)\times K$.
	As $K_\ap\subset K$ and $K$ is connected, by Lemma \ref{L1-cpt}\,(i) we get that  
	$Z(K_\ap)\subset Z(K)=\{e\}$ for each $\alpha$. Therefore, $K_\ap$ has trivial center and 
	$K_\ap$ are connected (cf.\ Lemma \ref{L1-cpt}\,(vi)). Now $K=Z^0_K(K_\ap)\times K_\ap$ and we get that 
	$$
	G=Z_G^0(K)\times K=Z_G^0(K)\times K'_\ap\times K_\ap,
	$$ 
	where $K'_\ap:=Z^0_K(K_\ap)$ is a connected semisimple Lie group with trivial center for each $\ap$. 
	Therefore, $\pi_\ap(G)$ is isomorphic to the direct product $Z_G^0(K)\times K'_\ap$, and hence 
	$\pi_\ap(C)=C_\ap$ is isomorphic to $C'\times C'_\ap$, where $C'$ (resp.\ $C'_\ap$) is a Cartan subgroup of $Z_G^0(K)$ 
	(resp.\ $K'_\ap$). As $K'_\ap$ is semisimple and has trivial center, as shown in the proof of Theorem \ref{existence-cartan}, 
	$C'_\ap$ is abelian. Therefore, we have that the length of the central series of each $C_\ap$ is same as that of $C'$. 
	This implies that $C$ is nilpotent. It is easy to see that $C$ is a maximal nilpotent group in $G$, since each $\pi_\ap(C)$ is 
	so in $G_\ap$. Suppose $L\subset C$ is a closed 
	normal subgroup of $C$ such that $C/L$ is compact and totally disconnected. Then for each $\ap$, 
	$\pi_\ap(C)/\pi_\ap(L)$ is isomorphic to $C/L(C\cap K_\ap)$, which is a quotient of $C/L$. 
	Therefore $\pi_\ap(C)/\pi_\ap(L)$ is compact and totally disconnected, and being a Lie group, it is finite. Hence, 
	$N_{G_\ap}(\pi_\ap(L))/\pi_\ap(L)$ is finite. Since $\pi_\ap(N_G(L))\subset N_{G_\ap}(\pi_\ap(L))$, we get that 
	$\pi_\ap(N_G(L))/\pi_\ap(L)$ is finite for each $\ap$. 
	Therefore, $N_G(L)/L$ is compact and $L^0\subset N_G^0(L)\subset L^0K_\ap$ for all $\ap$. Hence $N_G^0(L)=L^0$ 
	and $N_G(L)/L$ is totally disconnected. Thus, $C$ is a Cartan subgroup of $G$. 
	
	\smallskip
	\noindent {(2\mbox{b}):} Let $\{C_\ap\}_{\ap\in\Lambda}$ and $C=\cap_{\ap\in\Lambda}\,\pi_\ap^{-1}(C_\ap)$ be as in the hypothesis. 
	As each $C_\ap$ is a Cartan subgroup of $G_\ap$, to show that $C$ is a Cartan subgroup of $G$, by (2a), 
	it is enough to show that $\pi_\ap(C)=C_\ap$ for each $\ap\in\Lambda$. Let $\ap\in\Lambda$. By (1), there exists a Cartan 
	subgroup $C'_\ap$ in $G$ such that $\pi_\ap(C'_\ap) = C_\ap$. By Lemma \ref{L1-cpt}, we have that $G=Z_\ap K^0_\ap$, with 
	$Z_\ap=Z^0_G(K_\ap)$. Here $Z_\ap$ and $K^0_\ap$ are normal in $G$ and they centralise each other, and 
	$Z_\ap\cap K^0_\ap\subset Z(K^0_\ap)\subset Z(G)$ (cf.\ Lemma \ref{L1-cpt}\,(i)). Hence we get that $C'_\ap=C_{\ap 1}C_{\ap 2}$, where 
	$C_{\ap 1}=C'_\ap\cap Z_\ap$ (resp.\ $C_{\ap 2}=C'_\ap\cap K^0_\ap$) is a Cartan subgroup of $Z_\ap$ (resp.\ $K^0_\ap$).  
	Since $C'_\ap K_\ap=C_{\ap 1}K_\ap$ and since $C_{\ap 1}\cap K_\ap\subset Z_\ap\cap K_\ap\subset Z(K_\ap)\subset Z(G)$,  
	the following holds: For any Cartan subgroup $L$ of $G$ such that $\pi_\ap(L)=C_\ap$, we 
	get that $L\cap Z_\ap=C_{\ap 1}$. 
	
	We put a natural order on the index set $\Lambda$ as follows: For $\ap,\beta\in\Lambda$, $\ap\leq \beta$ if 
	$K_\ap\subset K_\beta$. Let $\ap,\beta\in\Lambda$ be such that $\ap\leq\beta$. Then $\pi_{\ap\beta}\circ\pi_\ap=\pi_\beta$ 
	and $\pi_\beta(C'_\ap)=\pi_{\ap\beta}\pi_\ap(C'_\ap)=\pi_{\ap\beta}(C_\ap)=C_\beta$. Hence we get that 
	$C'_\ap\cap Z_\beta=C_{\beta 1}$, i.e.\ $C_{\beta 1}\subset C'_\ap$. 
	Moreover, $C'_\ap K_\ap\subset C'_\ap K_\beta= C'_\beta K_\beta$,
	i.e.\ $\pi_\ap^{-1}(C_\ap)\subset\pi_\beta^{-1}(C_\beta)$. This implies that $C=\cap_{\ap\leq\beta}\, \pi_\ap^{-1}(C_\ap)$ 
	for each $\beta\in \Lambda$. 
	Now as $C_{\beta1}\subset C'_\ap$ for each $\ap\leq\beta$, we get that $C_{\beta 1}\subset C$. Therefore, 
	$C'_\beta K_\beta=C_{\beta 1}K_\beta\subset CK_\beta$. Conversely, $CK_\beta\subset\pi_\beta^{-1}(C_\beta)=C'_\beta K_\beta$. 
	Therefore, $\pi_\beta(C)=C_\beta$ for each $\beta\in\Lambda$. Hence, $C$ is a Cartan subgroup of $G$ by (2a). 
\end{proof}

In a connected locally compact group $G$, there exists a largest (closed) connected nilpotent normal subgroup which is known 
as the nilradical of $G$; this is well-known, it can also be deduced using the facts that it is true in case of a connected Lie group, the maximal 
compact normal subgroup $K$ of the radical $R$ in $G$ is abelian (cf.\ \cite{I}) and central in $R$ (cf.\ Lemma \ref{L1-cpt}\,(vii)), and $R/K$, being 
a Lie group,  admits a nilradical, and hence $R$ admits a nilradical, which is also the nilradical of $G$. Note that the maximal compact (central) 
subgroup $T$ of the nilradical $N$ of $G$ is connected and central in $G$ and $N/T$ is simply connected and nilpotent (cf.\ \cite{Gl}, Lemma 6.8). 
Hence, as in Proposition 3.1  of \cite{MS}, one can have a recipe for constructing a Cartan subgroup containing a nilpotent subgroup which 
complements the nilradical in some sense in a connected  locally compact solvable group. However, one can easily deduce the following 
directly by using the same result on Lie groups in \cite{MS} and Lemma \ref{central-cartan}. 

\begin{cor} \label{solv} Let $G$ be a connected locally compact solvable group and let $N$ be the nilradical of $G$. 
	If $L$ is a (closed) nilpotent subgroup of $G$ such that $G=LN$, then there exists a Cartan subgroup $C$ of $G$ 
	such that $L\subset C$. 
\end{cor}

\begin{proof}
	Let $K$ be the maximal compact normal subgroup of $G$. Since $G$ is solvable and Lie projective, it is easy to see that 
	$K$ is abelian (cf.\ \cite{I}). By Lemma \ref{L1-cpt}\,(i), $K$ is central in $G$, and hence, it is contained in every Cartan subgroup of $G$. 
	Let $\pi:G\to G/K$ be the natural projection. Then $G/K$ is a Lie group with $\pi(N)$ as its nilradical, and $G/K=\pi(L)\pi(N)$, where 
	$\pi(L)$ is nilpotent. Now by Proposition 3.1 of \cite{MS}, $\pi(L)$ is contained in a Cartan subgroup (say) $C'$ of $G/K$. Therefore,
	$L\subset LK\subset \pi^{-1}(C'):=C$ which is a Cartan subgroup of $G$ by Lemma \ref{central-cartan}\,(2).
\end{proof}

The following example show that in Corollary \ref{solv}, the inclusion \( L \subset C \) can be strict. It also shows that the
Cartan subgroups need not be abelian.

\begin{exa} \label{non-ab1}
	Let $G=\R_+^*\ltimes N$, where $N$ is a simply connected nilpotent group consisting of $5\times 5$ real strictly 
	upper triangular matrices, and the $\R_+^*$-acton on $N$ is as follows: For $t\in R_+^*$ and $a\in N$, where
	$$a= \begin{bmatrix}
		1 & a_{12} & a_{13} & a_{14} &a_{15}\\
		0 & 1 & a_{23}&a_{24}&a_{25} \\
		0 & 0 & 1&a_{34}&a_{35}\\
		0&0&0&1&a_{45}\\
		0&0&0&0&1
	\end{bmatrix},$$ 
	we set	
	$$t\cdot a = \begin{bmatrix}
		1 & ta_{12} & a_{13} & ta_{14} &a_{15}\\
		0 & 1 & t^{-1}a_{23}&a_{24}&t^{-1}a_{25} \\
		0 & 0 & 1&ta_{34}&a_{35}\\
		0&0&0&1&t^{-1}a_{45}\\
		0&0&0&0&1
	\end{bmatrix}.$$
	Clearly, $$Z_N(\mathbb R_+^*)=\begin{bmatrix}
		1 & 0 & a_{13} & 0 &a_{15}\\
		0 & 1 & 0&a_{24}& 0 \\
		0 & 0 & 1& 0&a_{35}\\
		0&0&0&1& 0\\
		0&0&0&0&1
	\end{bmatrix}.$$ 
	For $L:=\R_+^*$, $G:=LN$, we have that $C := LZ_N(L)$ is a Cartan subgroup of $G$ and it contains $L$ strictly. Moreover, $C$ is not abelian. 
	
	One can take $N$ to be a group of $\,n\times n$ unipotent real matrices, for any $n=2k+3$, $k\in\N$, and construct higher dimensional 
	connected solvable Lie groups of the form $LN$ with non-abelian Cartan subgroups.
\end{exa}

Using Example \ref{non-ab1}, one can also construct connected locally compact solvable groups, which are not Lie groups, with non-abelian 
Cartan subgroups; see Example \ref{non-ab2}. However, Cartan subgroups of a compact connected groups are abelian. 
The following is known for connected Lie groups.

\begin{prop}\label{abelian-conj}
	Let $G$ be a connected locally compact group such that $G/R$ is compact, where $R$ is the radical of $G$.  
	Then Cartan subgroups of $G$ are connected and conjugate to each other. Moreover, if $G$ is compact, then its Cartan 
	subgroups are precisely the maximal connected abelian subgroups.
\end{prop}

\begin{proof} 
	As $G$ is Lie projective, $G=\varprojlim_\ap G_\alpha$ for Lie groups  $G_\ap=G/K_\ap$, where compact groups $K_\ap$ and 
	$\pi_\ap: G\to G_\ap$ are as in Remark \ref{Kap}. Let $C$ be a Cartan subgroup of $G$.  By Proposition \ref{mod-cpt}, $\pi_\alpha(C)$ 
	is a Cartan subgroup of $G_\ap$ for each $\ap$. Note that  $C=\varprojlim_\ap \pi_\alpha(C)$. 
	Now $\pi_\ap(R)$ is the radical of $G_\alpha$ and the condition on $G$ implies that $G_\ap/\pi_\ap(R)$ is compact,
	and hence $\pi_\alpha(C)$ is connected for each $\ap$ (cf.\ \cite{Wi}, Proposition 6). Hence $C$ is connected. 
	
	Now we prove that all Cartan subgroups of $G$ are conjugate to each other. This holds in case 
	$G$ is a Lie group (cf.\ \cite{Wi}, Proposition 6). Now suppose $G$ is not a Lie group. 
	Note that all the conjugates of $C$ are also Cartan subgroups of $G$. Let $C'$ be any Cartan subgroup of $G$. We will 
	show that it is conjugate to $C$. Let $K$ be the largest compact normal subgroup of $G$. Then $K$ is nontrivial and 
	$K=Z(K)D(K^0)$ (see Lemma \ref{L1-cpt}(iv)). Since $Z(K)\subset Z(G)$ is contained in all Cartan subgroups, by 
	Lemma \ref{central-cartan}, we can replace $G$ by $G/Z(K)$, and $K$ by $K/Z(K)$, and assume that 
	$G=Z^0_G(K)\times K$, $K=D(K^0)=\Pi_d(R_d)$, where $Z^0_G(K)$ is a connected Lie group and each $R_d$ is a compact 
	simple Lie group with a trivial center.  Now $C=C_1\times C_2$ and $C'=C'_1\times C'_2$ where $C_1$ and $C'_1$ 
	(resp.\ $C_2$ and $C'_2$) are Cartan subgroups of $Z^0_G(K)$ (resp.\ $K$).  As the Cartan subgroups of $K$ are product 
	of Cartan subgroups of connected compact Lie groups $R_d$, we have that $C_2$ and $C'_2$ are conjugate. Now the 
	radical $R$ of $G$ is contained in $Z^0_G(K)$ (cf.\ Lemma \ref{L1-cpt}\,(vii)), it is also the radical of $Z^0_G(K)$, and $Z^0_G(K)/R$ 
	is compact. Hence we get that $C_1$ and $C'_1$ are also conjugate. This implies that $C$ and $C'$ are conjugate. 
	
	Now suppose $G$ is compact. Then so is $G_\ap$, and $\pi_\ap(C)$ is a maximal connected abelian subgroup of $\pi_\ap(G)$ for 
	each $\ap$, and hence $C$ is a maximal (connected) abelian subgroup of $G$. Conversely, it is easy to see that if $A$ is a 
	maximal connected abelian subgroup in $G$, then so is $\pi_\ap(A)$ in $\pi_\ap(G)$, and it is a Cartan 
	subgroup of $\pi_\ap(G)$ for each $\ap$. By Proposition \ref{mod-cpt}\,(2a), $A$ is a Cartan subgroup of $G$. 		
\end{proof}

For the connected Lie group $G=\SO(2,\R)\ltimes\R^2$, where $\SO(2,\R)$ is the special orthogonal group in $\GL(2,\R)$, the Cartan subgroups 
of $G$ are conjugates of $\SO(2,\R)$. For the connected Lie group $\R_+^*\ltimes\R$, all conjugates of $\R_+^*$ are its Cartan subgroups. 
In general, Cartan subgroups need not be connected. Note that any conjugate of a Cartan subgroup is also a Cartan subgroup in any 
connected locally compact group. However, they need not be connected or conjugate to each other. For example, Cartan subgroups of 
$\SL(2,\R)$ are either conjugates of $\SO(2,\R)$, which is connected and compact, or they are conjugates of the group $A$ of all diagonal 
matrices in $\SL(2,\R)$ and $A$ is non-compact and has two connected components. Now we give examples of Cartan subgroups of 
connected solvable and non-solvable locally compact groups which are not Lie groups. 

\begin{exa} \label{non-Lie} Consider the $3$-dimensional Heisenberg group $\Hm$; it is a unipotent group consisting of $3\times 3$ strictly upper 
	triangular real matrices. There is a continuous group action of $\GL(2,\R)$ on $\Hm$ by automorphisms and $\SL(2,\R)$ acts trivially on the 
	center $Z$ of $\Hm$. Let $\psi$ be an automorphism of $\SL(2,\R)\ltimes\Hm$ which acts trivially on $\SL(2,\R)$ and its action on $\Hm$ is 
	given by $2I\in\GL(2,\R)$, i.e.\ for $(a,b,c)\in\Hm$, where $c$ is the central element in $\Hm$, $\psi(a,b,c)=(2a,2b,4c)$. Note that the center 
	$Z$ of $\Hm$ is isomorphic to $\R$ and for a discrete infinite subgroup $D$ of $Z$, which is isomorphic to $\Z$, the action of $\SL(2,\R)$ 
	on $\Hm$ factors through $D$ and it acts on $\Hm/D$ by automorphisms, and it acts trivially on $Z/D$. Let $H=\SL(2,\R)\ltimes\Hm/D$ 
	(which is isomorphic to $(\SL(2,\R)\ltimes\Hm)/D$). Let $\psi_1$ be the homomorphism of $H$ induced by $\psi$, i.e.\ 
	$\psi_1$ acts trivially on $\SL(2,\R)$ and $\psi_1(xD)=\psi(x)D$, $x\in \Hm/D$. Note that $\psi_1$ is well-defined since $\psi(D)\subset D$. 
	Moreover, $\psi_1$ is continuous and surjective,  and $\ker\psi_1$ is a finite subgroup of order $4$ in $Z/D$, 
	the center of $H$, which is compact. The quotient map induced by $\psi_1$ on $H/(Z/D)=\SL(2,\R)\ltimes\R^2$ is the 
	automorphism induced by $\psi$. There is a closed subgroup $G$ 
	of $H^\Z$ defined as follows: $G:=\{(x_n)_{n\in\Z}\in H^\Z\mid x_{n+1}=\psi_1(x_n), n\in\Z\}$. 
	Then the center $K$ of $G$ is a compact subgroup of $(Z/D)^\Z$ (which is a compact 
	subgroup of $H^\Z$); $K$ is in fact a solenoid. Moreover, $G/K$ is isomorphic to $\SL(2,\R)\ltimes \R^2$, a Lie group, and hence 
	$G/K$ is locally compact. As $K$ is compact, it follows that $G$ is locally compact (cf.\ \cite{HRo}, Theorem 5.25).  Thus, $G$ is a 
	connected locally compact group which is not a Lie group. In fact, $G$ is isomorphic to $\SL(2,\R)\ltimes((\Hm/D)^\Z\cap G)$. As 
	Cartan subgroups of $G/K=\SL(2,\R)\ltimes \R^2$ are conjugates of either the 
	diagonal subgroup $A$ of $\,\SL(2,\R)$ or $\SO(2,\R)$ in $G/K$, by Lemma \ref{central-cartan}, Cartan subgroup of $G$ are conjugates of either 
	$A\times K$ or $\SO(2,\R)\times K$ in $G$. We can also take the closed subgroup $B:=\SO(2,\R)\ltimes((\Hm/D)^\Z\cap G)$ of $G$, then 
	$B$ is connected, locally compact and solvable, but not a Lie group, and all the Cartan subgroups of $B$ are conjugates of 
	$\SO(2,\R)\times K$ (Proposition \ref{abelian-conj}. Here, all the Cartan subgroups of $B$ (resp.\ $G$) are abelian.
\end{exa} 

\begin{exa} \label{non-ab2} There are connected locally compact solvable groups $G$ with the properties that $G$ is not a Lie group, the 
	center of $G$ is a solenoid and the Cartan subgroups of $G$ are non-abelian. Such a group $G$ can be constructed using the connected 
	solvable Lie group $LN$ as in Example \ref{non-ab1}. Let $Z$ be the center of $LN$ which is isomorphic to $\R$, $D$ be a discrete infinite 
	subgroup of $Z$ and let $H=(LN)/D$. It is easy to see that $LN$ admits automorphisms $\psi$ which act trivially on $L$ and the 
	$\psi$-action on $N$ is such that $\psi(z)=k^4z$, $z\in Z$, for any $k\in\N\mi\{1\}$. Using a similar construction as in Example \ref{non-Lie} 
	involving $\psi$, one can get a connected locally compact solvable group $G\subset H^\Z$ with the properties mentioned above such 
	that it has a non-abelian Cartan subgroup. Moreover, by Proposition \ref{abelian-conj}, the Cartan subgroups of $G$ are conjugate to each other, 
	hence they are all non-abelian. 
\end{exa}

Now we list some properties of Cartan subgroups of a connected locally compact group $G$. It is well-known that any closed solvable 
subgroup of $G$ is compactly generated (see Main Theorem of \cite{HN}), and hence all Cartan subgroups 
of $G$ are compactly generated. 
The following is known for Lie groups (see \cite{MS}, Corollary 3.7, Lemma 3.9 and Remark 3.10). 

\begin{lem} \label{properties} Let $G$ be a connected locally compact group and let $C$ be a Cartan subgroup of $G$. 
	Let $R$ and $N$ be respectively the radical and the nilradical of $G$. Then the following hold:
	\begin{enumerate}
		\item[{$(1)$}]  $N_G(C^0)=N_G(C)$, $N_R(C^0)=N_R(C)=C\cap R$ and $N_N(C^0)=N_N(C)=C\cap N$. 
		\item[{$(2)$}]  $Z_G(C^0)\subset C$. Moreover, if $G$ is semisimple, then $Z_G(C^0)=C$. 
		\item[{$(3)$}]  If $H$ is a closed connected subgroup of $G$ such that $C\subset H$, then $C$ is a Cartan subgroup of $H$. 
	\end{enumerate}
\end{lem}

In Lemma \ref{properties}, statements in (1) and (2) can be deduced easily from the same statements for Lie groups 
in \cite{MS} and Proposition \ref{mod-cpt}, and (3) follows easily from the definition of Cartan subgroups. We will not go into details.

Now we give an equivalent criteria to determine Cartan subgroups in connected locally compact groups, which is known for Lie groups 
(see Corollary 3.2 and Remark 3.10 of \cite{MS}).

\begin{prop} \label{cartan-equi-def} Let $G$ be a connected locally compact group. Then 
	$C$ is a Cartan subgroup of $G$ if and only if $C$ is a maximal nilpotent subgroup of $G$ and 
	$N_G(C^0)/C$ is compact and totally disconnected. If $G$ is solvable, then $C$ is a Cartan subgroup 
	of $G$ if and only if $C$ is connected and nilpotent, and $N_G(C)=C$. 
\end{prop}

\begin{proof} Let $C$ be a maximal nilpotent subgroup of $G$. Let $L$ be any closed normal subgroup of $C$ such that 
	$C/L$ is compact and totally disconnected. Then $L^0=C^0$ and $N_G(L)\subset N_G(C^0)$. Suppose $N_G(C^0)/C$ is compact 
	and totally disconnected. Then so is $N_G(C^0)/L$, and hence also, $N_G(L)/L$. Thus $C$ is a Cartan subgroup of $G$.  
	The converse follows easily as $N_G(C)=N_G(C^0)$ by Lemma \ref{properties}\,(1); which also implies that $N_G(C^0)/C$ is a 
	group. Now suppose $G$ is solvable. By Proposition \ref{abelian-conj}, any Cartan subgroup $C$ of $G$ is connected, and by 
	Lemma \ref{properties}\,(1), we have it $N_G(C)=C$. The converse in this case can easily be deduced using 
	Proposition \ref{mod-cpt} and Corollary 3.2 of \cite{MS}.
\end{proof}

	\section{Cartan subgroups in quotient groups}

In \cite{MS}, for a connected Lie group $G$, Cartan subgroups of quotient groups of $G$ are characterised as 
precisely the images of Cartan subgroups of $G$ (see Theorem 1.5 in \cite{MS}). The following generalises the 
same from Lie groups to connected locally compact groups.

\begin{theorem} \label{quotient} Let $G$ a connected locally compact group and let $H$ be a closed normal subgroup of $G$. 
	Then the following hold:
	\begin{enumerate}
		\item[{$(1)$}]  If $C$ is a Cartan subgroup of $G$, then $CH/H$ is a Cartan subgroup of $G/H$.
		\item[{$(2)$}]  If $C'$ is a Cartan subgroup of $G/H$, then there exists a Cartan subgroup of $C$ of $G$ such that $CH/H=C'$. 
	\end{enumerate}
\end{theorem}

\begin{proof} $(1):$ This can be proven easily using Theorem 1.5\,(a) of \cite{MS} and Proposition \ref{mod-cpt}. We 
	give a proof for the sake of completeness. Let $G=\varprojlim_\ap G_\alpha$, where $G_\alpha=G/K_\alpha$, 
	$\{K_\alpha\}$ and $\pi_\alpha:G\to G_\alpha$ are as in Remark \ref{Kap}. Let $C$ be a Cartan subgroup of $G$. By 
	Proposition \ref{mod-cpt}, $\pi_\ap(C)$ is a Cartan subgroup of $G_\ap$, for each $\ap$. As $\pi_\ap(H)$ is closed 
	and normal subgroup of $G_\ap$, by Theorem 1.5 of \cite{MS}, $(\pi_\ap(C)\pi_\ap(H))/\pi_\ap(H)$ is a Cartan 
	subgroup of $G_\ap/\pi_\ap(H)$, i.e.\ $CHK_\ap/HK_\ap$ is a Cartan subgroup of $G/HK_\ap$ for each $\ap$. 
	Since $CH/H$ (resp.\ $G/H$) is a projective limit of $CHK_\ap/HK_\ap$ (resp.\ $G/HK_\ap$), 
	by Proposition \ref{mod-cpt}\,(2a), we get that $CH/H$ is a Cartan subgroup of $G/H$. 
	
	\smallskip
	\noindent $(2):$ Suppose $C'$ is a Cartan subgroup of $G/H$. Let $\pi:G\to G/H$ be the natural projection. For the 
	largest compact normal subgroup $K$ of $G$, $Z(K)\subset Z(G)$, $\pi(Z(K))$ is central in $\pi(G)$ and it is 
	contained in $C'$ (cf.\ Lemmas \ref{mod-cpt} and \ref{central-cartan}). Note that $G/HZ(K)$ is a quotient of $G/H$ 
	whose kernel is $\pi(Z(K))$ and by (1), $C'/\pi(Z(K))$ is 
	a Cartan subgroup of $G/HZ(K)$. Now by Lemma \ref{central-cartan}, we may replace $G$, $H$ and $C'$ by $G/Z(K)$, 
	$HZ(K)/Z(K)$ and $C'/\pi(Z(K))$ respectively and assume that $K$ is connected, semisimple and it has trivial center 
	(see Lemma \ref{L1-cpt} \,(vi)). Then $K$ is a product of connected simple Lie groups, $Z^0_G(K)$ a connected Lie group 
	which does not have any nontrivial compact normal subgroup and $G=Z^0_G(K)\times K$. Since $G$ is connected and 
	$H$ is normal in $G$, we have that $H/H^0$ is totally disconnected and normal in $G/H^0$. Therefore, $H/H^0$ is central 
	in $G/H^0$, and by Lemma \ref{central-cartan}\,(2), there exists a Cartan subgroup in $G/H^0$ whose image in $G/H$ is 
	$C'$. Now we may assume that $H$ is connected. Let $M$ be the largest compact normal subgroup of $H$. Then $M$, 
	being characteristic in $H$, is normal in $G$. Therefore, $M\subset K$ and $H\cap K=M$. By Proposition \ref{mod-cpt}\,(1), 
	we may replace $G$ and $H$ by $G/M$ and $H/M$ respectively and assume that $H$ is a Lie group. 
	Now $H\cap K=\{e\}$, and as $H$ is normal in $G$, we get that $H\subset Z^0_G(K)$, which is a Lie group. Now  
	$G/H$ is isomorphic to $(Z^0_G(K)/H)\times K$, and we can choose a Cartan subgroups $C'_1$ of $Z^0_G(K)/H$ 
	and $C_2$ of $K$ such that for $C'=C'_1\times C_2$. By Theorem 1.5 of \cite{MS}, there 
	exists a Cartan subgroup $C_1$ of $Z^0_G(K)$ such that $C_1H/H=C'_1$. Let $C=C_1\times C_2$.
	Then $CH/H=C'$. 
\end{proof}

For a connected solvable locally compact group $G$, by Lemma 6.7 of \cite{Gl}, $\ol{[G,G]}$ is nilpotent, and hence for the nilradical $N$ of $G$, 
$G/N$ is abelian and a Cartan subgroup of itself. Using Theorem \ref{quotient} and Corollary \ref{solv}, the following can be deduced easily. 

\begin{cor} \label{c-solv} Let $G$ be a connected locally compact solvable group with the nilradical $N$. Then $C$ is a Cartan subgroup of $G$ if and
	only if $C$ is a maximal nilpotent group and $G=CN$.
\end{cor}

Using Theorem \ref{quotient}, Propositions \ref{abelian-conj} and \ref{cartan-equi-def}, we now derive a result about 
the structure of Cartan subgroups in a connected locally compact group which is an arbitrary product of locally compact 
groups, all but finitely many of which are compact.

\begin{prop} \label{product}
	Let $G$ be a connected locally compact group such that $G=\Pi_{\alpha\in \Lambda} G_\alpha$ with 
	the product topology {\rm (}$\Lambda$  is an arbitrary indexing set{\rm )}, where all but finitely many $G_\alpha$'s are 
	compact. Then the following hold:
	\begin{enumerate} 
		\item[{$(i)$}] Any Cartan subgroup $C$ of $G$ is of the form $C=\Pi_{\alpha\in\Lambda} \psi_\alpha(C)$, where 
		$\psi_\alpha:G\to G_\alpha$ is the natural projection and $\psi_\alpha(C)$ is a Cartan subgroup 
		of $G_\alpha$ for each $\alpha\in\Lambda$. 
		\item[{$(ii)$}] If $C_\alpha$ is a Cartan subgroup of $G_\alpha$ for each $\alpha\in\Lambda$, then 
		$C=\Pi_{\alpha\in\Lambda} C_\alpha$ is a Cartan subgroup of $G$. 		
	\end{enumerate}
\end{prop}

\begin{proof}
	$(i):$	 Let $C$ be any Cartan subgroup of $G$. Then $C$ is nilpotent, 
	and hence $\psi_\alpha(C)$ is nilpotent and $l(\psi_\alpha(C))\leq n=l(C)$. Therefore $\Pi_{\alpha\in\Lambda}\psi_\alpha(C)$ 
	is also nilpotent, and as it contains  the Cartan subgroup $C$ which is a maximal nilpotent group in $G$, we have that 
	$C=\Pi_{\alpha\in\Lambda}\psi_\alpha(C)$. Note that each $G_\alpha$ is isomorphic to $G/\ker\psi_\alpha$. Hence by 
	Theorem \ref{quotient}, $\psi_\alpha(C)$ is Cartan subgroup of $G_\alpha$ for each $\alpha$. 
	
	\smallskip
	\noindent $(ii):$ Let $C_\alpha$ be a Cartan subgroup of $G_\alpha$, $\alpha\in\Lambda$, and let 
	$C=\Pi_{\alpha\in\Lambda} C_\alpha$ as in the hypothesis. Since all but finitely many $G_\alpha$'s are compact, 
	we have by Proposition \ref{abelian-conj} that all but finitely many $C_\alpha$'s are abelian. Therefore, 
	$\{l(C_\alpha)\mid \alpha\in \Lambda\}$ is finite, and hence $C$ is nilpotent. Moreover, $C$ is a maximal nilpotent 
	subgroup in $G$ as each $C_\alpha$ is so in $G_\alpha$. By Proposition \ref{cartan-equi-def}, it is enough to show 
	that $N_G(C^0)/C$ is compact and totally disconnected.  Since we have that $C^0=\Pi_{\alpha\in\Lambda} C^0_\alpha$, 
	it follows that $N_G(C^0)=\Pi_{\alpha\in\Lambda} N_{G_\alpha}(C^0_\alpha)$, and $N_G(C^0)/C$ is isomorphic to 
	$\Pi_{\alpha\in\Lambda} (N_{G_\alpha}(C^0_\alpha)/C_\alpha)$. As $C_\alpha$ is a Cartan subgroup of $G_\alpha$, 
	by Proposition \ref{cartan-equi-def}, $N_{G_\alpha}(C^0_\alpha)/C_\alpha$ is compact and totally disconnected for each 
	$\alpha$. Therefore, $N_G(C^0)/C$ is compact and totally disconnected. Hence $C$ is a Cartan subgroup of $G$. 
\end{proof}

The following lemma which is known for Lie groups will be useful for the proof of Theorem \ref{decomp}.

\begin{lem} \label{rad-nil} Let $G$ be a connected locally compact group and let $R$ and $N$ be respectively the radical 
	and the nilradical of $G$. Let $C$ be a Cartan subgroup of $G$. Then $R=(C\cap R)N$ and both $C\cap R$ and $C\cap N$ are connected. 
\end{lem}

\begin{proof} We first show that $R/N$ is central in $G/N$. This is known if $G$ is a Lie group. Let $\{K_\ap\}$, Lie groups 
	$G_\ap=G/K_\ap$  and $\pi_\ap:G\to G_\ap$ be as in Remark \ref{Kap}. As $\pi_\ap(R)$ and $\pi_\ap(N)$ are respectively the radical 
	and the nilradical of $G/K_\ap$, we get that $R\subset \{x\in G\mid xgx^{-1}g^{-1}\in NK_\ap\mbox{ for all } g\in G\}$. Since 
	this is true for all $K_\ap$, we get that $R/N$ is central in $G/N$. By Theorem \ref{quotient}, $CN/N$ is a Cartan subgroup of $G/N$, 
	and hence it contains $R/N$. Therefore, $R \subset CN$, i.e.\ $R=(C\cap R) N$, which proves the first assertion. 
	Since $R$ is connected, we get that $R=(C\cap R)^0N$, and hence that $C\cap R=(C\cap R)^0(C\cap N)$. Therefore, 
	$C\cap R$ is connected if $C\cap N$ is connected. 
	
	Now we show that $C\cap N$ is connected. If $G$ is compact, then $N=Z^0(G)\subset C$, and hence $C\cap N=N$ is connected. 
	Suppose $G$ is not compact. Let $K$ be the largest compact normal subgroup of $G$. By Lemma \ref{L1-cpt}, 
	$K=Z(K)D(K^0)$, $Z(K)\subset Z(G)$ and $G=G_1D(K^0)=D(K^0)G_1$, where $G_1:=Z^0_G(D(K^0))$. As 
	$G_1\cap D(K^0)$ is central in $D(K^0)$, and hence in $G$,  we can use Proposition \ref{mod-cpt} and 
	Lemma \ref{central-cartan} and get that $C=C_1C_2=C_2C_1$, where $C_1$ 
	(resp.\ $C_2$) is a Cartan subgroup of $G_1$ (resp.\ $D(K^0)$). Note that since $D(K^0)$ is semisimple, 
	its nilradical is trivial and it is easy to see that $N$ is the nilradical of $G_1$. Since 
	$C\cap G_1=C_1$, we get that $C\cap N=C_1\cap N$, and hence without loss of any generality we may 
	replace $G$ by $G_1$ and assume that the maximal compact normal subgroup $K$ of $G$ is abelian, 
	and hence, central in $G$. 
	
	Let $K_\ap$ and $\pi_\ap:G\to G/K_\ap$ be as above. 
	As $K$ is central in $G$, so is each $K_\ap$, and hence $K_\ap\subset C$. By Proposition \ref{mod-cpt}, we have that 
	$\pi_\ap(C)$ is a Cartan subgroup of the Lie group $G/K_\ap$. As $\pi_\ap(N)$ is the nilradical of $G/K_\ap$, we have  
	that $\pi_\ap(C) \cap \pi_\ap(N)$ is connected (cf.\ \cite{Wu2}, Theorem 1.9). Now 
	$CK_\ap\cap NK_\ap=(CK_\ap\cap N)K_\ap=(C\cap N)K_\ap$. Therefore, $C\cap N=(C\cap N)^0(K_\ap\cap N)$. 
	Since this is true for all $\ap$ and since $\cap_\ap K_\ap=\{e\}$, we get that $C\cap N=(C\cap N)^0$ and it is connected. 
\end{proof}

\section{Levi decompositions and Cartan subgroups}

In \cite{Mat}, a Levi decomposition for a connected locally compact group $G$ is obtained; namely, $G = SR$, where 
$R$ is the radical of $G$ and $S$ is either trivial or it is a Levi subgroup, i.e.\ a maximal connected semisimple subgroup of $G$.  
Note that $S$ is not closed in general. Recall that a connected locally compact group is said to be  semisimple, if it has no proper 
connected abelian (solvable) normal subgroup, i.e.\ its radical is trivial. Equivalently, a connected subgroup $S$ of $G$ is semisimple 
if for all $\ap$, $\pi_\ap(S)$ is 
semisimple in the Lie group $\pi_\ap(G)$, where $\pi_\ap:G\to G/K_\ap$ and $K_\ap$ are as in Remark \ref{Kap}. 
Also, $G=S_0K_0R$, $S=S_0 K_0=K_0S_0$, where $S_0$ is a connected semisimple Lie group without any nontrivial 
compact connected normal subgroup and $K_0$ is a compact connected semisimple group, $S_0\cap K_0$ is a finite central 
subgroup and $[S_0,K_0]=\{e\}$ (see \cite{Mat}). It is also shown in \cite{Mat}, that all Levi subgroups are conjugate to each 
other by elements of the radical $R$. Note that a Levi subgroup $S$ of $G$, being a product of a connected (semisimple) 
Lie group and a connected compact (semisimple) group, is locally compact. Let $K$ be any compact normal subgroup 
of $G$. Then $S\cap K$ is normal in $S$. Recall that $K=Z(K)D(K^0)$ and $D(K^0)$ is a normal semisimple subgroup of 
$G$. Therefore $D(K^0)$ is contained in $S$, and hence $D(K^0)\subset K_0$. Now we show that $S\cap K$ is compact. 
It is enough to show that $Z(K)\cap S$ is compact. In fact, 
$Z(K)K_0\cap S_0$ is central in $S$ and as $S_0$ is a Lie group, it is discrete and being closed in $S_0$ (cf.\ \cite{Rago}), it is 
compact and finite in $S_0$. Therefore, $K_0\cap Z(K)$ is a subgroup of finite index in $S\cap Z(K)$. Thus $S\cap Z(K)$, 
and hence $S\cap K$ is compact. It follows that $S$ is a projective limit of Lie groups $S/(S\cap K_\ap)$ which are isomorphic 
to $\pi_\ap(S)$. In fact, for all small $K_\ap$, $S_0\cap K_\ap=\{e\}$ and $S\cap K_\ap\subset K_0$. 
Therefore, $\pi_\ap(S)$ is isomorphic to $S_0(K_0/K_\ap)$ for all small $K_\ap$. 

It has been shown by W\"ustner that any Cartan subgroup $C$ in a connected Lie group has a `Levi' decomposition, 
i.e.\ there exists a Levi subgroup $S$ such that $C=(C\cap S)(C\cap R)$, where $C\cap S$ is a Cartan subgroup of $S$ 
which centralises $C\cap R$. Here, we extend this result to all connected locally compact groups $G$. Note that any 
Levi subgroup $S$ of $G$ is locally compact as mentioned above, and we have that $S=S_0K_0$, an almost direct 
product of a compact group $K_0$ and a Lie group $S_0$, and Cartan subgroups of $S$ are 
products of Cartan subgroups of $S_0$ and $K_0$. We need the following fact: for any Levi subgroup $S$, the 
group $[S,R]$ generated by $\{srs^{-1}r^{-1}\mid s\in S,\, r\in R\}$ is contained in the nilradical $N$ of $G$; this can be 
deduced by using the same assertion for connected Lie groups and Lie projectivity of $G$. We also need the following:

\begin{lem} \label{td-center} Let $G$ be a connected locally compact group and let $D$ be a totally disconnected compact 
	central subgroup of $G$. Let $G=SR$ be a Levi decomposition with a Levi subgroup $S$ and the radical $R$. Then any 
	$d\in D$ can be decomposed as $d=d_sd_r$, where $d_s\in Z(S)$ and $d_r\in Z_R(S)$, the centraliser of $S$ in $R$. 
	In particular, $d_s$ belongs to every Cartan subgroup of $S$. 
\end{lem} 

\begin{proof} Let $d=d_sd_r$, where $d_s\in S$ and $d_r\in R$. Then $d_rd_s=d_sd_r$. For $x\in S$, we have 
	$e=dxd^{-1}x^{-1}=d_rd_sxd_s^{-1}d_r^{-1}x^{-1}$. Hence $d_sxd_s^{-1}x^{-1}=d_r^{-1}xd_rx^{-1}\in S\cap R$ for every 
	$x\in S$. Now $A=\{d_sxd_s^{-1}x^{-1}\mid x\in S\}$ is connected as $S$ is so. As $S\cap R$ is a normal solvable 
	subgroup of $S$ which is semisimple, $S\cap R$ is totally disconnected, and we get that $A=\{e\}$ and hence 
	$d_s\in Z(S)$ and $d_r\in Z_R(S)$. Therefore, $d_s$ belongs to every Cartan subgroup of $S$. 
\end{proof}

Generalising W\"ustner's decomposition theorem for Cartan subgroups in Lie groups and Theorem 1.1 of \cite{MS}, we get a 
`Levi' decomposition of Cartan subgroup in the following theorem. Recall that for subgroups $H$ and $M$ of $G$,  the group  
$Z_H(M)=\{x\in H\mid xyx^{-1}=y\mbox{ for all }y\in M\}$ is the centraliser 
of $M$ in $H$, and $Z^0_H(M)=(Z_H(M))^0$.

\begin{theorem} \label{decomp}
	Let $G$ be a connected locally compact group and let $R$ be the radical of $G$. 
	\begin{enumerate}
		\item[{$(1)$}]  Given a Cartan subgroup $C$ of $G$, there exists a Levi decomposition 
		$G=SR$ such that $C=(C\cap S)(C\cap R)$, where $C\cap S$ is a Cartan 
		subgroup of $S$ and $C\cap S$ centralises 
		$C\cap R$. Moreover, $Z_R(C\cap S)$ is connected and $C\cap R$ is a Cartan subgroup of $Z_R(C\cap S)$. 
		\item[{$(2)$}]  Given any Levi subgroup $S$ of $G$ and a Cartan subgroup $C_S$ of $S$, 
		$Z_R(C_S)$ is connected and for any Cartan subgroup 
		$C_{Z_R(C_S)}$ of it, $C_SC_{Z_R(C_S)}$ is a Cartan subgroup of $G$.
	\end{enumerate}
\end{theorem} 

\begin{proof} $(1):$ If $G$ is solvable, then $S$ is trivial and the assertions follow easily. Suppose $G$ is not solvable. 
	If $G$ is compact, then $G=D(G)Z^0(G)$, where $D(G)=[G,G]$ is a compact connected 
	semisimple group. Note that $R=Z^0(G)$, which is contained in $C$. In particular $C\cap R$ is connected and 
	central in $C$. Now $C=(C\cap D(G)) Z^0(G)$ and using properties of Cartan subgroups, it is easy to see that $C\cap D(G)$ is 
	a Cartan subgroup of $D(G)$. Here $D(G)$ is the unique Levi subgroup of $G$. Moreover, 
	$C\cap R=Z^0(G)=R=Z_R(C\cap D(G))$. Thus (1) holds in this case. 
	
	\smallskip 
	\noindent {\bf Step 1:} We now prove (1) for a Cartan subgroup $C$ in a general connected locally compact group $G$. 
	It has been shown to hold when $G$ is compact. Now suppose $G$ is non-compact. 
	Suppose the maximal compact normal subgroup $K$ of $G$ has trivial center. By Lemma \ref{L1-cpt}, $K=D(K^0)$ and it is 
	a compact connected semisimple normal group contained in every Levi subgroup. Then $G=L\times K$, where $L:= 
	Z^0_G(K)$ is a connected Lie group as $L\cap K=Z(K)=\{e\}$. Hence, $C=C_L\times C_K$, where $C_L$ (resp.\ $C_K$) 
	is a Cartan subgroup of $L$ (resp.\ $K$). Note that the radical $R$ of $G$ is same as the radical of $L$ since 
	$K$ is semisimple and its radical is trivial.
	
	Now from W\"ustner's decomposition theorem, $L$ has a Levi decomposition $L=S_LR$ such that 
	$C_L=(C_L\cap S_L)(C_L\cap R)$, where $(C_L\cap S_L)$ is a Cartan subgroup of a Levi subgroup $S_L$ of $L$ 
	and it centralises $(C_L\cap R)$. As $K$ centralises $L$, 
	we have that $G=S_LKR=SR$, where $S=S_LK=KS_L$ is a Levi subgroup of $G$ and 
	$C=C_LC_K=(C_L\cap S_L)(C_L\cap R)C_K=(C_L\cap S_L)C_K(C_L\cap R)$. Hence $C= (C\cap S)(C\cap R)$. 
	
	Here, $(C_L\cap S_L)C_K$ is a Cartan subgroup of $S$ and it is contained in $C\cap S$, which is nilpotent. Therefore,
	$C\cap S=(C_L\cap S_L)C_K$. As $C_K$ centralises $R$ and $C_L\cap S_L$ centralises $C_L\cap R$, we have that 
	$C\cap S$ centralises $C\cap R$. This proves that the first part of statement (1) holds when $Z(K)$ is trivial. 
	
	\smallskip
	\noindent {\bf Step 2:} Suppose $Z(K)\ne\{e\}$. Then $Z(K)$ is central in $G$ (cf.\ Lemma \ref{L1-cpt}). 
	Let $\rho:G\to G/Z(K)$ be the natural projection. Then $G':=\rho(G)$ is a connected 
	locally compact group and by Lemma \ref{central-cartan}, $C'=\rho(C)$ is a Cartan subgroup of $G'$. By Lemma \ref{L1-cpt}, 
	$K':=\rho(K)$ is the maximal compact normal subgroup of $G'$ with trivial center. By Step 1, there exists a Levi decomposition 
	$G'=S'R'$, where $S'$ is a Levi subgroup of $G'$, $R'=\rho(R)$ is the radical of $G'$ and $C'=(C'\cap S')(C'\cap R')$, where 
	$C'\cap S'$ is a Cartan subgroup of $S'$ and it centralises $(C'\cap R')$. Now since all Levi subgroups of $G$ (resp.\ $G'$) are conjugate
	to each other (cf.\ \cite{Mat}) and their images in $G'$ are also Levi subgroups, there exists a Levi subgroup $S$ of $G$ such that 
	$\rho(S)=S'$. Since $\ker\rho=Z(K)$ is central in $G$, it is easy to see that $S$ is a unique Levi subgroup whose image is $S'$. Now 
	$\rho^{-1}(C')\cap \rho^{-1}(S')=(C\cap S)Z(K)$ and $\rho^{-1}(C'\cap R')=(C\cap R)Z(K)$. Therefore, $C=(C\cap S)(C\cap R)Z(K)$.
	Hence $C=(C\cap S)(C\cap R)D$ as $Z^0(K)\subset C\cap R$, where $Z(K)=Z^0(K)D$ for some compact totally disconnected 
	central subgroup $D$ of $G$ (cf.\ \cite{HoMo1}, Theorem 9.41). For any $d\in D$, by Lemma \ref{td-center}, we have that 
	$d=d_sd_r=d_rd_s$ with $d_s\in Z(S)$ and $d_r \in Z_R(S)$, and $d_s$ belongs to every Cartan subgroup of $S$. 
	Thus if $C\cap S$ is a Cartan subgroup of $S$, then $d_s\in C\cap S$, and since $D\subset C$, it 
	further implies that  $d_r\in C\cap R$ and that $C=(C\cap S)(C\cap R)$.
		
	Now we show that $C\cap S$ is a Cartan subgroup of $S$. Let $\rho_1:S\to S'$ be the restriction of $\rho$ to $S$. Since 
	$\ker\rho_1=Z(K)\cap S$ is central in $S$ and  $\rho_1(S)=S'$, we get that $\rho_1^{-1}(C'\cap S')$ is a Cartan subgroup 
	of $S$. From the above, we get that $\rho_1^{-1}(C'\cap S')\subset \rho^{-1}(C')\cap S=C\cap S$. Conversely, 
	$\rho_1(C\cap S)\subset \rho(C)\cap \rho(S)=C'\cap S'$. Therefore, $C\cap S$ is a Cartan subgroup of $S$. 
	
	\smallskip
	\noindent {\bf Step 3:} We show that $C\cap S$ centralises $C\cap R$ when $Z(K)\ne \{e\}$. (Note that it is already shown 
	in Step 1 when $Z(K)=\{e\}$). As noted above $\rho(C\cap S)$ centralises $\rho(C\cap R)$, and we get that   
	$$
	C\cap R\subset \{x\in R\mid xcx^{-1}c^{-1}\in Z(K) \mbox{ for all } c\in C\cap S\}.$$  
	Since $C\cap R$ is connected (cf.\ Lemma \ref{rad-nil}), we get that 
	$$
	C\cap R\subset \{x\in R\mid xcx^{-1}c^{-1}\in Z^0(K) \mbox{ for all } c\in C\cap S\}.$$ 
	
	Let $\{K_\ap\}$, Lie groups $G_\ap:=G/K_\ap$ and $\pi_\ap:G\to G_\ap$ be as in Remark \ref{Kap}. By Proposition \ref{mod-cpt}, 
	$\pi_\ap(C)=\pi_\ap(C\cap S)\pi_\ap(C\cap R)$ is a Cartan subgroup of $G_\ap$. Moreover, $S_\ap:=\pi_\ap(S)$ 
	is a Levi subgroup of $G_\ap$ and $R_\ap:=\pi_\ap(R)$ is the radical of $G_\ap$. As $S_\ap$ is isomorphic to 
	$S/(S\cap K_\ap)$, we also have that $\pi_\ap(C\cap S)$ is a Cartan subgroup of $S_\ap$. Now from the above, we get 
	that the image of $\pi_\ap(C\cap S)$ in $G_\ap/\pi_\ap(Z^0(K))$ centralises $\pi_\ap(C\cap R)/\pi_\ap(Z^0(K))$. As 
	$\pi_\ap(Z^0(K))$ is a connected central subgroup of $G_\ap$ and $C\cap R$ is connected, by Proposition 3.4 
	of \cite{MS},  $\pi_\ap(C\cap S)$ centralises $\pi_\ap(C\cap R)$. Therefore, 
	$C\cap R\subset \{x\in R\mid xcx^{-1}c^{-1}\in K_\ap \mbox{ for all } c\in C\cap S\}$ for all $\ap$. As 
	$\cap_\ap K_\ap=\{e\}$, we get that $C\cap R\subset Z_R(C\cap S)$, i.e.\  $C\cap S$ centralises $C\cap R$. 
	
	\smallskip
	\noindent {\bf Step 4:} Now we show that $Z_R(C\cap S)$ is connected. This is known if $G$ is a Lie group (cf.\ \cite{MS}). 
	From the above, $C\cap R\subset Z_R(C\cap S)$. By Lemma \ref{rad-nil}, we have that $C\cap R$ is connected and 
	$R=(C\cap R)N$. Therefore,  $Z_R(C\cap S)=(C\cap R)Z_N(C\cap S)$. As $C\cap R$ is connected, it is enough 
	to show that $Z_N(C\cap S)$ is connected. This is already known in case $G$ is a connected Lie group. 
	
	Let $\psi:G\to G/K$ be the natural projection, where $K$ is any compact normal subgroup of $G$ such that   
	$G/K$ is a Lie group. Now $G/K=S'R'$ is a Levi decomposition with $S':=\psi(S)$ as a Levi subgroup of $G/K$, and 
	$R':=\psi(R)$ and $N':=\psi(N)$ are the radical and the nilradical respectively of $G/K$. As shown earlier, 
	$S\cap K$ is closed in $S$ and it is compact, and $S'$ is isomorphic to $S/(S\cap K)$. By Proposition \ref{mod-cpt}, 
	$\psi(C\cap S)$ is a Cartan subgroup of $S'$. By Proposition 3.6 of \cite{MS}, $Z_{N'}(\psi(C\cap S))$ is connected. Let 
	$$
	A_K:=\{x\in N\mid xcx^{-1}c^{-1}\in K \mbox{ for all }c\in C\cap S\}.$$ 
	It is easy to see that $A_K=\psi^{-1}(Z_{N'}(\psi(C\cap S)))\cap N$ and $\psi(A_K)=Z_{N'}(\psi(C\cap S))$, 
	and since the latter group is connected and $K\cap N\subset A_K$, we have that  $A_K=A_K^0(K\cap N)$. 
	Note that $K\cap N$ is contained in the unique compact subgroup (say) $M$ of $N$. As $M$ is connected and central 
	in $G$, it is contained in $A^0_K$, and hence $A_K$ is connected. For $\{K_\ap\}$ as in Remark \ref{Kap}, we have that 
	$A_{K_\ap}$ is connected for each $\ap$ and $Z_N(C\cap S)=\cap_\ap A_{K_\ap}$. Now it is enough to prove that 
	$(\cap_\ap A_{K_\ap})/M=\cap_\ap (A_{K_\ap}/M)$ is connected in $N/M$ which is a simply connected nilpotent Lie group. 
	This follows from a general statement that in a simply connected nilpotent Lie group, an arbitrary intersection of connected 
	subgroups is connected. 
	
	\smallskip
	\noindent {\bf Step 5:} Now we have that $Z_R(C\cap S)$ is connected and $C\cap R\subset Z_R(C\cap S)$. 
	We show that $C\cap R$ is a Cartan subgroup of $Z:=Z_R(C\cap S)$. Suppose $C\cap R\subset C'$, where $C'\subset Z$ is a 
	nilpotent group. Then $C\cap S$ and $C'$ centralise each other and $(C\cap S)C'=CC'$ is nilpotent and as it contains $C$, we 
	have that $CC'=C$. Therefore, $C'=C\cap R$ and $C\cap R$ is a maximal nilpotent subgroup of $Z$. We know that $C\cap R$ 
	is connected. Since $Z$ is solvable, in view of Proposition \ref{cartan-equi-def}, it is enough to show that $N_Z(C\cap R)=C\cap R$. Let 
	$H:=N_Z(C\cap R)$. As shown in Step 4, $Z=(C\cap R)(Z\cap N)$. Therefore,  
	$H=(C\cap R)(H\cap N)$. Since $H\cap N$ is nilpotent and normal in $H$ and it normalises $C\cap R$, we get that $H$ 
	itself is nilpotent and hence $C\cap R$, being a maximal nilpotent subgroup of $Z$, is the same as $H$. Therefore, $C\cap R$ 
	is a Cartan subgroup of $Z=Z_R(C\cap S)$. 
	
	\smallskip
	\noindent $(2):$ Let $S$ be a Levi subgroup of $G$ and let $C_S$ be a Cartan subgroup of $S$. We have a 
	Levi decomposition $G=SR$ for the radical $R$ of $G$. We know that if $G$ is a Lie group, then (2) holds by 
	Theorem 1.1 of \cite{MS}. Suppose $G$ is not a Lie group. If $G$ is compact, then $G=Z^0(G)D(G)$ and $S=D(G)$ is the 
	unique Levi subgroup of $G$ and $R=Z^0(G)$. Therefore $C_SZ^0(G)=C_SR$ is a Cartan subgroup of $G$ and $C_S$ 
	centralises $R=C\cap R=Z_R(C_S)$. 
	
	Now suppose $G$ is a not compact.  Let $K$ be the maximal compact normal subgroup of $G$ 
	such that $G/K$ is a Lie group and let $\pi:G\to G/K$ be the natural projection. 
	Then $G/K$ is a Lie group and $G/K=\pi(S)\pi(R)$ is a Levi decomposition, and $\pi(C_S)$ is a Cartan subgroup of
	$\pi(S)$ (this follows from Theorem \ref{quotient} as $S\cap K$ is compact and normal in $S$, and $\pi(S)$ is isomorphic to 
	$S/(S\cap K)$). Now by Theorem 1.1 of \cite{MS}, there exists a Cartan subgroup $C'$ of $G/K$ such that 
	$C'=\pi(C_S)(C'\cap \pi(R))$ and $C'\cap\pi(R)$ centralises $\pi(C_S)$ and it is a Cartan
	subgroup of $Z_{\pi(R)}(\pi(C_S))$ which is connected. 
	
	Now suppose $K$ is abelian. Then $K$ is central in $G$ and by Lemma \ref{central-cartan}, $C:=\pi^{-1}(C')$ 
	is a Cartan subgroup of $G$. Here, $C_S\subset C\cap S$ which is nilpotent, and since $C_S$ is a Cartan subgroup of $S$, 
	we get that $C_S=C\cap S$. Now $CR=C_SKR=C_SDR$, where $D$ is a compact totally disconnected central subgroup of $G$. 
	By Lemma \ref{td-center}, $D\subset Z(S)R\subset C_SR$, and hence $C\subset C_S(C\cap R)$. Therefore, $C=C_S(C\cap R)$. 
	By Lemma \ref{rad-nil}, $C\cap R$ is connected. 
	
	Note that $\pi(C_S)$ is a Cartan subgroup of $\pi(S)$ and it centralises $\pi(C)\cap \pi(R)$, and we have that 
	$C\cap R\subset \{x\in R\mid xcx^{-1}c^{-1}\in K \mbox{ for all } c\in C_S\}$. As $C\cap R$ is connected,
	we get that $C\cap R\subset \{x\in R\mid xcx^{-1}c^{-1}\in K^0 \mbox{ for all } c \in C_S\}$. 
	As $K$ is abelian, and hence, central in $G$, arguing as in Step 3, we get that $C\cap R\subset Z_R(C_S)$. 
	Now we can show as in Steps 4 and 5 that  $Z_R(C_S)$ 
	is connected and that $C\cap R$ is a Cartan subgroup of $Z_R(C_S)$. 
	
	Now suppose $K$ is not abelian. Then $K=Z(K)D(K^0)$, where $Z(K)\subset Z(G)$ and $D(K^0)$ is a connected compact 
	semisimple normal subgroup contained in every Levi subgroup. By Lemma \ref{L1-cpt}, $G=G_1D(K^0)=D(K^0)G_1$, where 
	$G_1=Z_G^0(D(K^0))$ and $S=S_1D(K^0)=D(K^0)S_1$, where $S_1$ is a Levi subgroup of $G_1$. Now 
	$C_S=C_{S_1}C_1=C_1C_{S_1}$ where $C_{S_1}$
	(resp.\ $C_1$) is a Cartan subgroup of $S_1$ (resp.\ $D(K^0)$). Note that the largest compact normal subgroup of $G_1$ is 
	central in $G_1$ (as well as in $G$) and the radical $R$ of $G$ is the radical of $G_1$. Now arguing as above, we get that 
	there exists a Cartan subgroup $C_{G_1}$ of $G_1$ such that $C_{G_1}\cap S_1=C_{S_1}$, $C_{G_1}\cap R$ centralises 
	$C_{S_1}$ as well as $C_1$, $Z_R(C_S)=Z_R(C_{S_1})$ is connected and $C_{G_1}\cap R$ is a Cartan subgroup of 
	$Z_R(C_S)$. Now $C=C_{G_1}C_1$ is a Cartan subgroup of $G$ and $C\cap S=C_{S_1}C_1=C_S$ and 
	$C\cap R=C_{G_1}\cap R$ has the desired property. 
	
	By Proposition \ref{abelian-conj}, any Cartan subgroup $C_{Z_R(C_S)}$ of $Z_R(C_S)$ is a conjugate of $C\cap R$ 
	by an element of $Z_R(C_S)$, hence $C_SC_{Z_R(C_S)}$ is a conjugate of $C$, and it is a Cartan subgroup of $G$. 
\end{proof}

Using Theorems \ref{quotient} and \ref{decomp} along with the structure of Levi subgroups in \cite{Mat}, we prove the following 
corollary about the intersection of Levi subgroup with the center $Z(G)$ of a locally compact group $G$; it is known for Lie 
groups. 

\begin{cor} \label{center-ss}
	Let $G$ be a  connected locally compact group and let $S$ be a Levi subgroup of $G$. Then $Z(S)/(S\cap Z(G))$ is 
	compact and totally disconnected. 
\end{cor}

\begin{proof} We have a Levi decomposition $G=SR$ for the radical $R$ of $G$. As $S\cap Z(G)$ is central in $S$, using 
	the structure of $S=S_0K_0$ as mentioned above and a result of \cite{Rago} on Lie groups, one can show that it is a 
	closed subgroup of $S$. Suppose $S$ is compact. Then $Z(S)$ is compact and totally disconnected, and 
	the assertion follows trivially. Now suppose $S$ is a non-compact Lie subgroup. Since $G/R$ is also a Lie group, 
	there exists a compact normal subgroup $K$ of $G$ such that $G/K$ is a Lie group, $K\subset R$ and $K\cap S=\{e\}$. 
	This implies that $K$ is abelian, and by Lemma \ref{L1-cpt}, it is central in $G$. Let 
	$$
	Z_K=\{x\in G\mid xgx^{-1}g^{-1}\in K \mbox{ for all } g\in G\}.$$
	Then $K\subset Z(G)\subset Z_K$. We now show that  $Z_K\cap S= Z(G)\cap S$. 
	Since $K$ is central in $G$, $Z_K$ centralises $[G,G]$, and hence also $\ol{[G,G]}$. As $S\cap K=\{e\}$ 
	(or as $S\subset [G,G]$), we get that $Z_K\cap S$ is central in $S$, i.e.\ $Z_K\cap S\subset Z(S)$, and hence 
	it is contained in every Cartan subgroup of $S$. Let $C_S$ be a Cartan subgroup of $S$. By Theorem \ref{decomp}, 
	there exists a Cartan subgroup $C$ of $G$ such that $C=C_S(C\cap R)$ and $C_S$ centralises $C\cap R$. 
	Therefore, $Z_K\cap S$ 
	centralises $C$. Let $\varrho:G\to G/\ol{[G,G]}$ be the natural projection. By Theorem \ref{quotient},
	$\varrho(C)$ is a Cartan subgroup of $\varrho(G)$. As $\varrho(G)$ is abelian $\varrho(C)=\varrho(G)$. 
	Therefore, $G=C\ol{[G,G]}$, and it follows that $Z_K\cap S$ centralises $G$, i.e.\ $Z_K\cap S\subset Z(G)$ and 
	$S\cap Z_K=S\cap Z(G)=Z(S)\cap Z(G)$. 
	
	Let $\pi:G\to G/K$ be the natural projection. Then $\pi(Z(S))=Z(\pi(S))$ as $K\cap S=\{e\}$. Also,  
	$Z_K=\pi^{-1}(Z(\pi(G)))$. As $\pi(G)$ is a Lie group, we get that $Z(\pi(G))\cap \pi(Z(S))$ has finite index in 
	$\pi(Z(S))$. Therefore, $Z_K\cap Z(S)K=(Z_K\cap Z(S))K=(Z(G)\cap S)K$ has finite index in 
	$Z(S)K$. Now as $K\cap S=\{e\}$, it follows that $S\cap Z(G)=Z(S)\cap Z(G)$ has finite index in $Z(S)$. 
	
	Now suppose $S$ is not compact and it is not a Lie group. By Theorem 1 of \cite{Mat}, $S=S_1L= LS_1$, where  
	$S_1\ne\{e\}$ is a connected Lie group without any compact factors, $L$ is the largest compact connected normal 
	(semisimple) subgroup of $S$, and $L$ and $S_1$ centralise each other. Now $Z(S)=Z(S_1)Z(L)$ where $Z(L)$ is compact
	and totally disconnected and $Z(S_1)$ is a discrete subgroup of $S_1$. Therefore, it is enough to show that $Z(S_1)\cap Z(G)$
	has finite index in $Z(S_1)$. Since $S_1R$ is a closed normal subgroup of $G$ (cf.\ \cite{Mat}, Theorem 1), it is locally compact 
	and as $S_1$ is a Levi subgroup of $S_1R$, we get from the above that $Z(S_1)\cap Z(S_1R)$ has finite index in $Z(S_1)$. 
	As $Z(S_1)$ is central in $S=S_1L$, we get that $Z(S_1)\cap Z(G)=Z(S_1)\cap Z(S_1R)$ has finite index in $Z(S_1)$. 
\end{proof}

The following generalises the same result from connected Lie groups to connected locally compact groups 
(see Lemmas 3.8 and 3.9\,(3) of \cite{MS}). 

\begin{cor} \label{centraliser-conn}
	Let $C$ be a Cartan subgroup of a connected locally compact group. Then the following hold:
	\begin{enumerate}
		\item[{$(1)$}]  $C=Z_G(C^0)C^0$. 
		\item[{$(2)$}]  $C/(Z(G)C^0)$ is finite.
	\end{enumerate}
\end{cor}

\begin{proof} 
	If $G$ is a connected Lie group, then (1) and (2) hold (cf.\ \cite{MS}, Lemmas 3.8 and 3.9\,(3)). 
	If $G$ is solvable, or more generally, if $G/R$ compact for the radical $R$ of $G$, then $C$ is connected 
	(cf.\ Proposition \ref{abelian-conj}), and hence both (1) and (2) hold trivially. 
	
	Now suppose $G$ is not a Lie group and $G/R$ is not compact. Let $G=SR$ be a Levi decomposition such
	that $C=(C\cap S)(C\cap R)$, where $C_S:=C\cap S$ is a Cartan subgroup of the Levi subgroup $S$ and $C\cap R$ 
	is connected and it centralises $C_S$ (cf.\ Theorem \ref{decomp}). By Lemma \ref{properties}\,(2),  $C_S=Z_S(C_S^0)$. 
	Therefore, as $C_S$ centralises $C\cap R$, we have that $C_S\subset Z_G(C^0)$, and hence 
	$C\subset Z_G(C^0)(C\cap R)=Z_G(C^0)C^0$. Conversely, by Lemma \ref{properties}\,(2), $Z_G(C^0)\subset C$, 
	and hence (1) holds. 
	
	Now we prove (2). By Theorem 1 of \cite{Mat}, $S=S_1L$, where $S_1\ne\{e\}$ is a connected 
	Lie group without any compact factors, $L$ is the largest compact connected normal (semisimple) subgroup of $S$, and $L$ 
	and $S_1$ centralise each other. For $C_S=C\cap S$ as above, we have that  $C_S=C_1C_L=C_LC_1$, where $C_1$ 
	(resp.\ $C_L$) is a Cartan subgroup of $S_1$ (resp.\ $L$). By Proposition \ref{abelian-conj}, $C_L$ is connected as $L $ is compact.
	Since $S_1$ is a Lie group, we know that $Z(S_1)C_1^0$ is a subgroup of finite index in $C_1$ and since  
	$C=C_1C_L(C\cap R)=C_1C^0$, we get that $C/(Z(S_1)C^0)$ is finite. Now we show that $Z(S_1)/(Z(S_1)\cap Z(G))$ is finite.
	We already know that $Z(S_1)$ centralises $L$. Since $S_1R$ is closed, it is locally compact. As $S_1$ is a Lie group, $Z(S_1)$ is 
	a discrete finitely generated abelian subgroup in $S$, and by Corollary \ref{center-ss}, $Z(S_1)\cap Z(S_1R)$ has finite index in $Z(S_1)$. 
	As $S_1$ centralises $L$, we get that $Z(S_1)\cap Z(S_1R)\subset Z(G)$ and $Z(S_1)\cap Z(G)$ has finite index in $Z(S_1)$. 
	Therefore, $C_1/(C_1^0(Z(S_1)\cap Z(G)))$ is finite. As $C=C_1C^0$ and $Z(G)\subset C$, from the above we get that 
	$C/(Z(G)C^0)$ is finite. 
\end{proof}

	\section{Cartan subgroups and maximal tori of the radical}

Here, we would like to equate Cartan subgroups of a connected locally compact group $G$ and those of 
centralisers in $G$ of maximal tori of the radical of $G$. Any connected locally compact group, being Lie projective, admits 
maximal compact subgroups which are connected, and conjugate to each other, 
and they are abelian if $G$ is solvable (cf.\ \cite{I} Theorem 13, Lemma 2.2). We observe that the maximal compact normal 
subgroup of the radical $R$ of $G$ is central in $G$  (by Lemma \ref{L1-cpt}). Following a terminology for Lie groups, we will call 
such a maximal compact subgroup in $R$ a maximal torus of $R$. 
If $G$ is a Lie group, it follows from \cite{Go} that every Cartan subgroup of $G$ contains a maximal torus of the radical $R$.  
It has been shown in Theorem 1.2 of \cite{MS}, that the centraliser in a connected Lie group $G$ of any maximal torus of $R$ 
is connected and its Cartan subgroups are Cartan subgroups of $G$; moreover, all Cartan subgroups arise in this way.
The following lemma will be useful for generalising Theorem 1.2 of \cite{MS} to all connected locally compact 
groups. 

\begin{lem} \label{tori} Let $G$ be a connected locally compact group, $K$ be a compact normal subgroup of $G$ and let 
	$\psi:G\to G/K$ be the natural projection. Let $R$ be the radical of $G$ and let $T_R$ be any maximal torus of $R$. 
	Then $K\subset Z_G(T_R)$, $\psi(T_R)$ is a maximal torus of the radical $\psi(R)$ of $\psi(G)$, 
	$\psi(Z_G(T_R))=Z_{\psi(G)}(\psi(T_R))$ and $Z_G(T_R)$ is connected. 
\end{lem}

\begin{proof} Note that for the radical $R$ of $G$, it is easy to see that $\psi(R)$ is the radical of $G/K$. Here,  
	$\psi(R)$ is isomorphic to $R/(R\cap K)$, where $R\cap K$, being abelian and normal in $G$, is central in $G$ and it is 
	contained in any maximal torus $T_R$, and hence $\psi(T_R)$ is a maximal torus in $\psi(R)$. The other assertions 
	follow easily if $T_R$ is central in $G$, as in that case, $Z_G(T_R)=G$ and $Z_{\psi(G)}(\psi(T_R))=\psi(G)$. 
	
	Now suppose $T_R$ is not central in $G$. By Lemma \ref{L1-cpt}\,(vii), we know that $R\subset Z_G(K)$. In particular, 
	$K\subset Z_G(R)\subset Z_G(T_R)$, and $Z_G(T_R)$ contains every compact normal subgroup of $G$.
	
	We now show that $\psi(Z_G(T_R))=Z_{\psi(G)}(\psi(T_R))$ and that $Z_G(T_R)$ is connected. It is easy to see   
	that $\psi(Z_G(T_R))\subset Z_{\psi(G)}(\psi(T_R))$. For any compact normal subgroup $M$ of $G$, let 
	$$
	A_M:=\{x\in G\mid xtx^{-1}t^{-1}\in M\mbox{ for all }  t \in T_R\}.$$ 
	Note that $M\subset A_M$ as $M$ is normal in $G$. Moreover, $Z_G(T_R)\subset A_M$, and hence, from the above, 
	we get that $A_M$ contains every compact normal subgroup of $G$. As $T_R$ is connected and $(M\cap R)^0=Z^0(M)$, 
	we get that 
	$$A_M=\{x\in G\mid xtx^{-1}t^{-1}\in Z^0(M)\mbox{ for all }  t \in T_R\}=A_{Z^0(M)}.$$ 
	
	We first assume that $G$ is a Lie group. As shown in the proof of Theorem 1.2 of \cite{MS}, we have that 
	$A_K=A_{Z^0(K)}=Z_G(T_R)Z^0(K)=Z_G(T_R)$. Therefore, $\psi(Z_G(T_R))=\psi(A_K)=Z_{\psi(G)}(\psi(T_R))$. Also,
	$Z_G(T_R)$ is connected (cf.\ \cite{MS}, Theorem 1.2). 
	
	Now suppose $G$ is not a Lie group. Let $\{K_\ap\}$, $G_\ap=G/K_\ap$ and  $\pi_\ap:G\to G_\ap$ be as in Remark \ref{Kap}. 
	As $G_\ap$ is a Lie group, and $Z^0(K)$  is central in $G$, we get that 
	\begin{eqnarray*}
		\pi_\ap(A_K)&=&\pi_\ap(A_{Z^0(K)})\cr
		&=&\{x\in G_\ap\mid xtx^{-1}t^{-1}\in \pi_\ap(Z^0(K))\mbox{ for all }t\in \pi_\ap(T_R)\}\cr
		&=&Z_{G_\ap}(\pi_\ap(T_R)).
	\end{eqnarray*}
	Therefore, $\pi_\ap(A_K)=\pi_\ap(A_{K_\ap})$, and hence 
	$A_K=A_K K_\ap=A_{K_\ap}$ for every $\ap$. Therefore, $A_K= \cap_\ap A_{K_\ap}=Z_G(T_R)$. Hence 
	$\psi(Z_G(T_R))=\psi(A_K)=Z_{\psi(G)}(\psi(T_R))$. Now $\pi_\ap(Z_G(T_R))=Z_{G_\ap}(\pi_\ap(T_R))$, and it 
	is connected for every $\ap$, hence we have that $Z_G(T_R)$ is connected.
\end{proof}

The following theorem is a generalisation of Theorem 1.2 of \cite{MS}, as we know from Lemma \ref{tori}, that the centraliser 
$Z_G(T_R)$ of $T_R$ in $G$ is connected, for any maximal compact subgroup $T_R$ of the radical $R$. It is known 
that any closed nilpotent subgroup of a connected locally compact group is compactly generated (cf.\ \cite{HN}) 
and hence has a unique maximal compact subgroup. For a Cartan subgroup $C$ of $G$, $C\cap R$ is normal in $C$ and 
its unique maximal compact subgroup is normal in $C$. We know that 
$C\cap R$ is connected (cf.\ Lemma \ref{rad-nil}), and it centralises $C\cap S$ where $C=(C\cap S)(C\cap R)$ for some 
Levi subgroup $S$ (cf.\ Theorem \ref{decomp}), which in turn imply that the maximal 
compact subgroup of $C\cap R$ is connected and central in $C$. 

\begin{theorem} \label{cartan-torus}  Let $G$ be a connected Lie group. Let $T_R$ be any maximal
	compact subgroup of the radical $R$. Then every Cartan subgroup of $Z_G(T_R)$ is a Cartan subgroup of $G$. 
	Conversely, for every Cartan subgroup $C$ of $G$, the following holds: if $M$ is the maximal compact $($connected$)$ 
	subgroup of $C\cap R$, then $M$ is the maximal torus of $R$ and $C$ is a Cartan subgroup of $Z_G(M)$.
\end{theorem}

\begin{proof} If $G$ is compact, then $R=T_R=Z^0(G)$ is compact and $G=Z_G(T_R)$, and the assertion holds trivially. Now 
	suppose $G$ is not compact. 
	Let $C'$ be a Cartan subgroup of $Z_G(T_R)$. Let $\{K_\ap\}$, Lie groups $G_\ap=G/K_\ap$ and $\pi_\ap:G\to G_\ap$ 
	be as in Remark \ref{Kap}. By Lemma \ref{tori}, $\pi_\ap(T_R)$ is a maximal torus in the radical $\pi_\ap(R)$ of $G_\ap$ and 
	$\pi_\ap(Z_G(T_R))=Z_{G_\ap}(\pi_\ap(T_R))$. By Proposition \ref{mod-cpt}\,(1), $\pi_\ap(C')$ is a Cartan subgroup of 
	$Z_{G_\ap}(\pi_\ap(T_R))$. Now by Theorem 1.2 of \cite{MS}, $\pi_\ap(C')$ is a Cartan subgroup 
	of $G_\ap$. Since this holds for every $\ap$, we get from Proposition \ref{mod-cpt}\,(2a) that $C'$ is a Cartan subgroup of $G$. 
	
	We now prove the converse. Let $C$ be a Cartan subgroup of $G$ and let $M$ be the maximal compact connected subgroup of 
	$C\cap R$. We first show that $M$ is a maximal torus in $R$. Let $\pi:G\to G/Z(K)$ be the natural projection, where $K$ is the 
	maximal compact normal subgroup of $G$. By Lemma \ref{L1-cpt}, $Z(K)$ is central in $G$, $\pi(K)$ is a compact connected 
	semisimple group with trivial center, and $\pi(G)=L\times \pi(K)$, where  $L:=Z^0_{\pi(G)}(\pi(K))$ is a Lie group. Moreover, 
	$\pi(R)\subset L$ and it is the radical of $\pi(G)$. By Lemma \ref{central-cartan}\,(2), $\pi(C)$ is a Cartan subgroup of $\pi(G)$, 
	and $\pi(C)=C_1\times C_2$, where $C_1$ (resp.\ $C_2$) is a Cartan subgroup of $L$ (resp.\ $\pi(K)$). As $L$ is a Lie group 
	and $\pi(C)\cap\pi(R)=C_1\cap\pi(R)$ is connected, and as $Z(K)\subset C$, we get that $\pi(C)\cap\pi(R)=\pi(C\cap R)$. 
	In particular, $\pi(M)$ is a maximal compact subgroup of $C_1\cap\pi(R)$, and it is a maximal torus of $\pi(R)$ 
	(cf.\ \cite{Go}, Theorem 3\,(ii)). As all the maximal tori in $R$ (resp.\ $\pi(R)$) are conjugate to each other,  by Lemma \ref{tori} 
	there exists a maximal torus $T_R$ in $R$ such that $\pi(T_R)=\pi(M)$. As $\ker\pi=Z(K)\subset C$, we get that 
	$T_R\subset C$, and hence $T_R=M$. Thus $M$ is a maximal torus of $R$. By Lemma \ref{tori}, $Z_G(M)$ is connected. 
	As $C\subset Z_G(M)$, $C$ is a Cartan subgroup of $Z_G(M)$ (cf.\ Lemma \ref{properties}\,(3)).
\end{proof}

\section{Cartan subalgebras}

The Lie algebra of a connected locally compact group can be defined as follows: 
A connected locally compact group $G$ is Lie projective and if $G=\varprojlim_\ap  G_\ap$ for Lie groups $G_\ap$, then the 
Lie algebra $L(G)$ of $G$ is defined as $L(G)=\varprojlim_\ap  L(G_\ap)$, the projective limit of $L(G_\ap)$, the Lie algebras 
of $G_\ap$, and it is independent of the choice of $\{G_\ap\}$ (see \cite{L} where it is called LP-Lie algebras). Note that $L(G)$ 
defined this way is isomorphic to the one defined in Ch.\ 2 of \cite{HoMo2}, where it is defined as the set of all continuous real
one-parameter subgroups of $G$ endowed with a Lie algebra structure and a topology of uniform convergence on compact sets. 
Since $L(G)$ is a projective limit of finite dimensional real Lie algebras, it is also called a pro-Lie algebra in \cite{HoMo2}.
A Lie algebra of a locally compact group $G$ is defined as the Lie algebra $L(G^0)$ of $G^0$. 

Note that if $\pi_\ap:G\to G_\ap$ denotes the natural projection and $\ker\pi_\ap=K_\ap$ is a compact normal subgroup 
of $G$ for each $\ap$, then $\du\pi_\ap:L(G)=\varprojlim_\ap L(G_\ap)\to L(G_\ap)$ is the corresponding projection  
(in terms of the definition of the Lie algebra in \cite{HoMo2}; if $X\in L(G)$ is a continuous real one-parameter 
subgroup of $G$, then $\pi_\ap\circ X$ is a continuous real one-parameter subgroup of $G_\ap$ and 
$\du\pi_\ap(X)=\pi_\ap\circ X\in L(G_\ap)$). Also, if $K_\ap\subset K_\beta$, then there exists a natural projection 
$\pi_{\ap\beta}:G_\ap\to G_\beta$ and the corresponding Lie algebra homomorphism 
$\du\pi_{\ap\beta}: L(G_\ap)\to L(G_\beta)$ such that $\pi_{\ap\beta}\circ\pi_\ap=\pi_\beta$
and $\du\pi_{\ap\beta}\circ\du\pi_\ap=\du\pi_\beta$.

A Lie algebra is said to be pro-nilpotent if it is a projective limit of nilpotent Lie algebras. A {\it Cartan subalgebra} 
of a pro-Lie algebra is a closed pro-nilpotent subalgebra that agrees with its own normaliser 
(see Definition 7.88 of \cite{HoMo2} for details, and for its existence see Theorem 7.93 of \cite{HoMo2}). 

In the following theorem, we prove that the Lie algebra of a Cartan subgroup exists and it is a Cartan subalgebra as defined 
in \cite{HoMo2}. Moreover, there is a one-to-one correspondence between Cartan subgroups and Cartan subalgebras. 
We will also show that every Cartan subalgebra of $L(G)$ is nilpotent. 

\begin{theorem}\label{cartan-subalgebra}
	Let $G$ be a connected locally compact group with the Lie algebra $\G$. Let $C$ be a Cartan subgroup of $G$. 
	Then the Lie algebra $L(C)$ of $C$ is a Cartan subalgebra of $L(G)$. Moreover, $L(C)$ is nilpotent. Conversely, 
	given a Cartan subalgebra  $\mathcal C$ of $L(G)$, there exists a unique Cartan subgroup $C$ of $G$ whose 
	Lie algebra is $\mathcal C$. In particular, all Cartan subalgebras of $L(G)$ are nilpotent. 
\end{theorem}

\begin{proof}
	Let $C$ be a Cartan subgroup of $G$. Let $L(C)$ be the subalgebra of $L(G)$ such that it is the Lie algebra of $C^0$. 
	As $G$ is a projective limit of Lie groups $G_\alpha=G/K_\alpha$ where $K_\alpha$ are compact normal subgroups of $G$, 
	by Proposition \ref{mod-cpt} we have $C=\varprojlim_\ap CK_\alpha/K_\alpha$. Since $C$ is a projective limits of Lie groups
	$C_\ap:= CK_\alpha/K_\alpha$, the Lie algebra of $C$ is a pro-Lie algebra, 
	i.e.\  $L(C)=\varprojlim_\ap L(C_\alpha)$,  where $L(C_\ap)$ is a Lie subalgebra of $L(G_\ap)$ for each $\ap$. Now $L(C)$ is 
	a Cartan subalgebra by Theorem 7.93 of \cite{HoMo2}, as $L(C_\alpha)$ is a Cartan subalgebra for each $\alpha$. 
	As $C$ is nilpotent, the length $l(C)$ of the central series of $C$ is finite. Then  
	$l(C_\ap^0)\leq l(C_\ap)\leq l(C)$ for all $\ap$. Therefore, $l(L(C_\ap))=l(C_\ap^0)\leq l(C)$ for every $\ap$, where 
	$l(L(C_\ap))$ is the length of the central series of the Lie subalgebras in $L(C_\ap)$, and hence 
	$L(C)$ is nilpotent.
	
	Conversely, let $\mathcal{C}$ be a Cartan subalgebra of $L(G)$. Then its projection $\mathcal{C}_\alpha$ in $L(G_\ap)$ is a 
	Cartan subalgebra of $L(G_\ap)$ for each $\ap$, and $\mathcal{C}=\varprojlim_\ap  \mathcal{C}_\alpha$ 
	(cf.\ \cite{HoMo1}, Proposition 7.90). For each $\ap$, there exists a unique Cartan subgroup $C_\alpha$ of $G_\ap$  
	whose Lie algebra is $\mathcal{C}_\alpha$ and $C_\alpha=Z_{G_\ap}(C_\ap^0)C_\ap^0$ (cf.\ \cite{MS}, Lemma 3.9\,(3), 
	see also Corollary \ref{centraliser-conn} above). For any pair $\ap,\beta$ such that $K_\ap\subset K_\beta$, 
	since $\du\pi_{\ap\beta}(\mathcal C_\ap) =\mathcal C_\beta$ and the exponential map is surjective for any connected 
	nilpotent Lie group, we have that $\pi_{\ap\beta}(C_\ap^0)=C_\beta^0$. Now 
	$\pi_{\ap\beta}(Z_{G_\ap}(C_\ap^0))\subset Z_{G_\beta}(C_\beta^0)$. Hence we have that $\pi_{\ap\beta}(C_\ap)\subset C_\beta$. 
	By Theorem \ref{quotient} (see also Theorem 1.5 in \cite{MS}), $\pi_{\ap\beta}(C_\ap)$ is a Cartan subgroup of $G_\beta$,  
	and hence it is equal to $C_\beta$. By Proposition \ref{mod-cpt}\,(2b), $C=\cap_\ap\,\pi_\ap^{-1}(C_\ap)$ is a Cartan subgroup 
	of $G$ and $\pi_\ap(C)=C_\ap$ for every $\ap$. Now $\mathcal{C}=\varprojlim_\ap\mathcal{C}_\alpha$, 
	where $\mathcal{C}_\alpha=L(C_\ap)=L(\pi_\ap(C))$ for each $\ap$. Hence we get that $\mathcal C$ is the Lie algebra of $C$. From 
	the earlier part of the proof, it follows that $\mathcal C$ is nilpotent. Now $\pi_\ap(C)=C_\ap$ is a unique Cartan subgroup in the Lie group 
	$G_\ap$ such that $L(C_\ap)=\mathcal{C}_\ap=\du \pi_\ap (\mathcal{C})$, hence we get that $C$ is a unique Cartan subgroup of $G$ 
	such that $L(C)=\mathcal{C}$.
\end{proof}

	\section{Power maps, Cartan subgroups and weak exponentiality}

Let $G$ be a group. For $k\in\N$, the $k$-th power map $P_k:G\to G$ is defined by $g\mapsto g^k$ for all $g\in G$. 
In this section, we characterise the dense image of any power map in a connected locally compact 
group in terms of Cartan subgroups.

Note that a connected locally compact group $G$ has compact normal subgroups $K_\ap$ such that $G_\ap=G/K_\ap$ 
are Lie groups and $G=\varprojlim_\ap  G_\ap$, see Remark \ref{Kap}.. Let $\pi_\ap:G\to G_\ap$ be the natural projections.  
A subset $D$ of $G$ is dense in $G$ if and only if $\pi_\ap(D)$ is dense in $G_\ap$ for each $\ap$ (this is easy to show, see for 
example a proposition on p.\ 327 in \cite{HoMo2}). The following lemma can be proven easily using the proposition 
mentioned above, as $\pi_\ap(P_k(G))=P_k(G_\ap)$ for each $\ap$. 

\begin{lem}\label{proj-power}
	Let $G=\varprojlim_\ap G/K_\alpha$, a projective limit of Lie groups $G_\alpha=G/K_\alpha$ where $K_\alpha$ 
	are compact normal subgroups in $G$. Let $k\in\N$.
	Then $P_k(G)$ is dense in $G$  if and only if $P_k(G_\alpha)$ is dense in $G_\alpha$ for every $\alpha$.
\end{lem}

We now extend Theorem 1.6 of \cite{MS} to any connected locally compact group $G$ as follows:

\begin{prop}[stability under extensions]  \label{pk-quo} Let $H$ be a closed normal subgroup of a connected locally 
	compact group $G$. Let $k\in \N$. If $P_k(H)$ is dense in $H$ and $P_k(G/H)$ is dense in $G/H$, then $P_k(G)$ is 
	dense in $G$.
\end{prop}

\begin{proof} Let $\{K_\ap\}$, Lie groups $G_\ap=G/K_\ap$ and $\pi_\ap:G\to G_\ap$ be as in Remark \ref{Kap}. 
	Then we have that $G/HK_\ap$ (resp.\ $HK_\ap/K_\ap$) is isomorphic to $(G/K_\ap)/(HK_\ap/K_\ap)$ 
	(resp.\ $H/(H\cap K_\ap)$). As $P_k(H)$ is dense in $H$, we get that $P_k(HK_\ap/K_\ap)$ is dense in 
	$HK_\ap/K_\ap$. Moreover, as $P_k(G/H)$ is dense in $G/H$, we get that $P_k(G/HK_\ap)$ is dense in 
	$G/HK_\ap$. Now by Theorem 1.6 of \cite{MS}, we get that $P_k(G/K_\ap)$ is dense in $G/K_\ap$. 
	In particular, $\pi_\ap(P_k(G))$ is dense in $G_\ap$ for each $\ap$. By Lemma \ref{proj-power}, $P_k(G)$ is dense in $G$. 
\end{proof}

We recall that there exists a closed connected characteristic subgroup $G_s$ of $G$ containing the radical $R$ 
such that $G_s/R$ is compact and $G/G_s$ is a semisimple Lie group without compact factors (see \cite{HoMu}). 
Note that $G_s$ is the largest connected normal amenable subgroup of $G$ and is known as the 
amenable radical of $G$.

Let $S(C)$ be the union of all Cartan subgroups of a connected locally compact group $G$. It is known that 
$S(C)$ is dense in $G$ if $G$ is a Lie group (see e.g.\ Propositions 1.5 and 1.6 in \cite{Ho}). It is also 
true for any connected locally compact group $G$, as shown in the proof of the following theorem, which 
correlates the density of the image of a power map $P_k$ with the surjectivity of its restriction to Cartan subgroups. 
The theorem below is essentially known for Lie groups. Note that any connected locally compact nilpotent group $N$ 
is divisible. Indeed, if $K$ is a maximal compact subgroup of $N$, then $K$ is connected and central in $N$ and 
$N/K$ is a (simply) connected nilpotent Lie group. Since both $K$ and $N/K$ are divisible, we get that $N$ is divisible.

\begin{theorem} \label{power-dense}
	Let $G$ be a connected locally compact group. Let $k\in \N$. Then the following are equivalent.
	\begin{enumerate}
		\item[{$(1)$}] $P_k(G)$ is dense in $G$.
		\item[{$(2)$}] $P_k(C)= C$ for every Cartan subgroup $C$ of $G$.
		\item[{$(3)$}] $P_k(C/C^0)=C/C^0$ for every Cartan subgroup $C$ of $G$. 
		\item[{$(4)$}] $P_k(S(C))=S(C)$.
		\item[{$(5)$}]  $P_k(G/K)$ is dense in $G/K$ for a compact connected normal subgroup 
		$K$ of $G$.
		\item[{$(6)$}]  $P_k(G/R)$ is dense in $G/R$, where $R$ is the radical of $G$.
		\item[{$(7)$}] $P_k(G/G_s)$ is dense in $G/G_s$, where $G_s$ is as above.
		\item[{$(8)$}] $P_k(G/H)$ is dense in $G/H$, where $H$ is a closed connected normal 
		subgroup of $G$ such that $R\subset H$ and $H/R$ is compact.
	\end{enumerate}
\end{theorem}

\begin{proof}
	Since $G$ is a connected locally compact group, we have $\{K_\ap\}$, connected Lie groups $G_\ap=G/K_\ap$, 
	$\pi_\ap:G\to G_\ap$ and  $G=\varprojlim_\ap G_\alpha$  as in Remark \ref{Kap}. 
	
	\smallskip
	\noindent $(1)\Rightarrow  (2):$ 
	Suppose that $P_k(G)$ is dense in $G$. Let $C$ be any Cartan subgroup of $G$. By Proposition \ref{mod-cpt}\,$(2a)$, 
	$C_\alpha:=\pi_\alpha(C)$ is a Cartan subgroup of $G_\alpha$. Now 
	$P_k(G_\alpha)$ is dense in $G_\ap$, and  by Theorem 1.1 of \cite{BM} we have $P_k(C_\alpha)= C_\alpha$ 
	for every $\ap$. Now $P_k(C)K_\ap=CK_\ap$ for all $\ap$. Since $C$ is compactly generated and nilpotent, 
	by Theorem 3.1.17 of \cite{Hey}, $C$ is strongly root compact. In particular, for any compact set $B\subset C$, 
	$B_r=\{y\in C\mid y^n\in B \mbox{ for some } n\in\N\}$ is relatively compact (cf.\ \cite{Hey}, Theorem 3.1.13). 
	Hence, it follows that $P_k(C)$ is closed. Therefore, $P_k(C)=C$. 
	
	\smallskip
	\noindent $(2)\Leftrightarrow (3):$ Note that $(2)\Rightarrow (3)$ is obvious. Let $C$ be any Cartan subgroup of $G$. 
	Suppose $(3)$ holds, i.e.\ $P_k(C/C^0)=C/C^0$. Since $C^0$ is connected and nilpotent, $P_k(C^0)=C^0$. Also, by 
	Corollary \ref{centraliser-conn}\,(1), $C=Z_G(C^0)C^0$. Then (3) implies that $P_k(Z_G(C^0))C^0=C$. As $C^0$ is   
	centralises $Z_G(C^0)$ we get that 
	$$
	P_k(C)=P_k(Z_G(C^0)C^0)=P_k(Z_G(C^0))P_k(C^0)=P_k(Z_G(C^0))C^0=C.$$ 
	Hence (2) holds. That is, $(3)\Rightarrow (2)$.
	
	\smallskip
	\noindent $(2)\Rightarrow  (4)$ is obvious. Now we show that $(4)\Rightarrow (1)$. Suppose $P_k(S(C))=S(C)$. It is enough 
	to show that $S(C)$ is dense in $G$. By Proposition \ref{mod-cpt}\,(1),  $\pi_\ap(S(C))$ is the union of 
	Cartan subgroups in  $G_\ap$, which is dense in $G_\ap$ for each $\ap$. Hence, by a proposition on p.\ 327 
	of \cite{HoMo2}, $S(C)$ is dense in $G$. Thus, $(1-4)$ are equivalent. 	
	
	\smallskip
	\noindent $(1)\Leftrightarrow (5):$ Here, $(1)\Rightarrow (5)$ is obvious. For the converse, suppose $(5)$ holds. 
	By Proposition \ref{abelian-conj}, the Cartan subgroup of $K$ is connected. Now using $(2)\Rightarrow (1)$ for
	$K$ instead of $G$, we get that $P_k(K)$ is dense in $K$. Now by Proposition \ref{pk-quo}, $P_k(G)$ is dense in $G$, 
	i.e.\ (1) holds. 
	
	\smallskip
	\noindent $(1)\Leftrightarrow (6):$ Note that $(1)\Rightarrow (6)$ is obvious. Since the Cartan subgroups of 
	the radical $R$ are connected by Proposition \ref{abelian-conj}, arguing as above, we get that $(6)\Rightarrow (1)$. 
	
	\smallskip
	\noindent $(1)\Leftrightarrow (7)$ and $(1)\Leftrightarrow (8)$ can be proven the same way as $(1)\Leftrightarrow (5)$ since
	the Cartan subgroups of $G_s$, as well as those of the subgroup $H$ as in $(8)$ are connected by Proposition \ref{abelian-conj}. 
\end{proof}

\begin{rem} As in the case of Lie groups, the density of the image of $P_k$ in a connected locally compact group is equivalent to 
	that of the restriction of $P_k$ to a Levi subgroup of it. More generally, if $G=SR$ is a Levi decomposition, with a Levi subgroup $S$, 
	where $S=S_0K_0=K_0S_0$, $S_0$ is a connected Lie group without any compact factors and $K_0$ is the largest compact 
	connected normal subgroup of $S$ (cf.\ \cite{Mat}), then $G_s$ as mentioned in Theorem \ref{power-dense}\,$(7)$ is the same as $K_0R$, 
	and it follows from Theorem \ref{power-dense} and Theorem \ref{decomp} that $P_k(G)$ is dense in $G$ if and only if $P_k(S_0)$ is 
	dense in $S_0$.
\end{rem}

An element $x$ in a connected locally compact group $G$ is said to be 
exponential if it is contained in a continuous real one-parameter subgroup of $G$. 
A subgroup $H$ of $G$ is said to be exponential (resp.\ weakly exponential), if every element in $H$ is exponential, 
(resp.\ the set of exponential elements in $H$ is dense in $H$). Any (weakly) exponential group is necessarily connected. 
It is known that any compactly generated locally compact nilpotent group $G$ is connected if and only if it is divisible. 
This is also equivalent to saying that such a $G$ is weakly exponential. More generally, we state the following 
which is known for Lie groups (see \cite{N} and \cite{BM}).

\begin{cor}\label{C-power-cartan}
	Let $G$ be a connected locally compact group. Then the following are equivalent.
	\begin{enumerate}
		\item[{$(i)$}] $G$ is weakly exponential.
		\item[{$(ii)$}] $P_k(G)$ is dense in $G$ for every $k\in\N$.
		\item[{$(iii)$}] Every Cartan subgroup of $G$ is connected.
	\end{enumerate}
\end{cor}

In the Corollary \ref{C-power-cartan}, $(i)\Rightarrow (ii)$ is obvious, $(ii)\Leftrightarrow (iii)$ follows easily from Theorem \ref{power-dense} 
and $(ii)\Rightarrow (i)$ follows from Lemma \ref{proj-power}, Corollary 1.3 of \cite{BM} and Lemma 3.2 of \cite{HoMu}. We will not go into 
details here. Note that every divisible element in a connected Lie group is exponential (cf.\ \cite{Mc}), and hence every connected 
nilpotent Lie subgroup is exponential. However, there are compact connected abelian groups which are not exponential. For example, 
the $p$-adic solenoid is a compact connected abelian group, but it is not exponential (see Examples 14.4 and 14.5 of \cite{HoMo2}, see 
also Examples \ref{non-Lie} and \ref{non-ab2}, where the respective center as well as Cartan subgroups are not exponential). 
Thus the condition (iii) in Corollary \ref{C-power-cartan} is equivalent to the condition that every Cartan subgroup is weakly exponential.

\medskip
\noindent{\bf Acknowledgements:} The authors would like to thank the anonymous referee for useful comments and suggestions which led to 
improvement in the presentation of the manuscript.

\vskip2mm

\bigskip

\end{document}